\documentclass[11pt,leqno]{amsart}
\usepackage{epsfig}
\usepackage{amssymb}
\usepackage{amscd}
\usepackage[matrix,arrow]{xy}

\newtheorem{theo}{Theorem}[section]
\newtheorem{lemma}[theo]{Lemma}
\newtheorem{defi}[theo]{Definition}
\newtheorem{prop}[theo]{Proposition}

\newtheorem{cor}[theo]{Corollary}
\newtheorem{remark}[theo]{Remark}

\newtheorem{example}[theo]{Example}
\numberwithin{equation}{section}

\def\coh{\operatorname{coh}}
\def\Qcoh{\operatorname{Qcoh}}

\def\bR{{\mathbf R}}
\def\bL{{\mathbf L}}
\def\bP{{\mathbf P}}

\def\pre-tr{\operatorname{pre-tr}}

\def\Hom{\operatorname{Hom}}

\def\gr{\operatorname{gr}}

\newcommand{\bbA}{{\mathbb A}}
\newcommand{\bbC}{{\mathbb C}}

\newcommand{\bbR}{{\mathbb R}}

\newcommand{\bbZ}{{\mathbb Z}}

\newcommand{\bbP}{{\mathbb P}}

\newcommand{\bbQ}{{\mathbb Q}}

\newcommand{\cY}{{\mathcal Y}}

\newcommand{\cF}{{\mathcal F}}

\newcommand{\cO}{{\mathcal O}}

\newcommand{\cM}{{\mathcal M}}

\newcommand{\cD}{{\mathcal D}}

\newcommand{\cA}{{\mathcal A}}
\newcommand{\cB}{{\mathcal B}}

\newcommand{\cC}{{\mathcal C}}

\newcommand{\cW}{{\mathcal W}}
\newcommand{\cU}{{\mathcal U}}

\newcommand{\cH}{{\mathcal H}}

\newcommand{\cX}{{\mathcal X}}
\newcommand{\cZ}{{\mathcal Z}}

\newcommand{\Perf}{\operatorname{Perf}}
\newcommand{\Sh}{\operatorname{Sh}}
\newcommand{\supp}{\operatorname{Supp}}
\newcommand{\Ker}{\operatorname{Ker}}

\newcommand{\Ext}{\operatorname{Ext}}

\newcommand{\id}{\operatorname{id}}

\newcommand{\Mod}{\operatorname{Mod}}
\newcommand{\op}{\operatorname{op}}

\newcommand{\pt}{\operatorname{pt}}

\newcommand{\Coker}{\operatorname{Coker}}
\newcommand{\qc}{\operatorname{qc}}

\newcommand{\an}{\operatorname{an}}
\newcommand{\et}{\operatorname{et}}
\newcommand{\Fr}{\operatorname{Fr}}
\newcommand{\Supp}{\operatorname{Supp}}

\title[Categorical resolutions, poset schemes and Du Bois singularities]
{Categorical resolutions, poset schemes and Du Bois singularities}

\author{Valery A.~Lunts}
\address{Department of Mathematics, Indiana University,
Bloomington, IN 47405, USA} \email{vlunts@indiana.edu}

\begin{document}

\begin{abstract}We introduce the notion of a poset scheme and study 
the categories of quasi-coherent sheaves on such spaces. 
We then show that smooth poset schemes may be used to obtain 
categorical resolutions of singularities for usual singular schemes. 
We prove that a singular variety $X$ possesses such a resolution if 
and only if $X$ has Du Bois singularities. Finally we show that the de Rham-Du Bois 
complex for an algebraic variety $Y$ may be defined using any smooth poset scheme which 
satisfies the descent over $Y$ in the classical topology.
\end{abstract}

\maketitle

\tableofcontents

\section{Introduction}
\subsection{Categorical resolutions}
There is a good notion of smoothness for a DG algebra $A.$ Namely,
$A$ is called {\it smooth} if it is a perfect DG $A^{\op}\otimes
A$-module. This notion is Morita invariant: if DG algebras A and B
are derived equivalent (i.e. there exists a DG $A^{\op}\otimes
B$-module $M,$ such that the functor $(-)\stackrel{\bL}{\otimes
}_AM:D(A)\to D(B)$ is an equivalence), then $A$ is smooth if and
only if $B$ is such. This allows one to define smoothness of derived
categories $D(A),$ and consequently of cocomplete triangulated
categories which possess a compact generator (and have an
enhancement). Examples of such categories are the derived categories
$D(X)$ of quasi-coherent sheaves on quasi-compact and separated
schemes X (see for example \cite{BoVdB}). The scheme $X$ is smooth
if and only if the category $D(X)$ is smooth in the above sense.

In the paper \cite{Lu2} we have introduced the concept of a
categorical resolution of singularities. Namely, given a DG algebra
$A,$ a categorical resolution of $D(A)$ is a pair $(B,M),$ where $B$
is a smooth DG algebra and $M$ is a DG $A^{\op}\otimes B$-module,
such that the functor $(-)\stackrel{\bL}{\otimes }_AM:D(A)\to D(B)$
is full and faithful on the subcategory of perfect DG $A$-modules.
The main result of \cite{Lu2} is the following theorem.

\begin{theo}\label{prevmain} Let $X$ be a separated scheme of finite type
over a perfect field $k.$  Then

a) There exists a classical generator $E\in D^b(cohX)$, such that
the DG algebra $A=\bR \Hom (E,E)$ is smooth and hence the functor
$$\bR Hom (E,-):D(X)\to D(A)$$
is a categorical resolution.

b) Given any other classical generator $E^\prime \in D^b(cohX)$ with
$A^\prime =\bR \Hom (E^\prime ,E^\prime)$, the DG algebras $A$ and
$A^\prime $ are derived equivalent.
\end{theo}

This theorem provides an intrinsic categorical resolution for
$D(X).$ This resolution has the flavor of Koszul duality. The
resolving DG algebra $A$ is Morita equivalent to its opposite
$A^{\op}$ and usually has unbounded cohomology.

\begin{example}\label{dualnumb} If in Theorem \ref{prevmain}
$X=Spec(k[\epsilon]/\epsilon ^2),$ and $E=k,$ then $A=k[t],$ where
$\deg (t)=1.$
\end{example}

We should note that the notion of categorical resolutions is
different from the usual resolution of singularities. Namely if $X$
is an algebraic variety and $\sigma :\tilde{X}\to X$ is its
resolution of singularities, then $\bL \sigma ^*:D(X)\to
D(\tilde{X})$ is a categorical resolution if and only if $X$ has
rational singularities. If $D(X)\to D(A)$ is a categorical
resolution (and the singularities of $X$ are not rational), we find
that the category $D(A)$ has a closer relation to $D(X)$ than
$D(\tilde{X}).$ Also one may consider categorical resolutions of
nonreduced schemes.

\medskip

\noindent{\bf Conjecture.} Let $X$ be a separated scheme of finite
type over a field. Then there exists a smooth DG algebra $A$ with
$H^i(A)=0$ for $\vert i\vert >>0$ and a functor $D(X)\to D(A)$ which
is a categorical resolution.

\medskip

\subsection{Smooth poset schemes and Du Bois singularities}
In this article we introduce a new class of smooth categories, which
are constructed by "gluing" the categories $D(X)$ for smooth schemes
$X.$ Namely, we consider {\it poset schemes $\cX$} which by
definition are diagrams of schemes $\{X_\alpha\} _{\alpha \in S}$
indexed by elements of a finite poset $S$ with a morphism $f_{\alpha
\beta}:X_\alpha \to X_\beta$ iff $\alpha \geq \beta.$ There is a
natural notion of a quasi-coherent sheaf on $\cX,$ which gives us
the abelian category $Qcoh \cX$ and its derived category $D(\cX).$
This derived category is cocomplete and has a compact generator (if
all schemes $X_\alpha$ are separated and quasi-compact). So
$D(\cX)\simeq D(A)$ for a DG algebra $A.$ The category $D(\cX)$ is
smooth if the poset scheme $\cX$ is smooth (i.e. all schemes
$X_{\alpha}$ are such). In any case the category $D(\cX)$ has a
natural semi-orthogonal decomposition with semi-orthogonal summands
$D(X_\alpha), \alpha \in S.$ In this last sense we consider $D(\cX)$
as a gluing of the categories $D(X_{\alpha})$ along the morphisms
$f_{\alpha \beta}.$

There is a natural notion of a morphism $\pi :\cX \to X$ from a
poset scheme $\cX$ to a scheme $X$ and the corresponding functor
$\bL \pi ^*:D(X)\to D(\cX).$ We say that $\pi $ is a categorical
resolution if $\cX$ is smooth and $\bL \pi ^*$ is a categorical
resolution. We prove the following theorem (=Theorem
\ref{posetres=DuBois}).

\begin{theo} Let $X$ be a reduced separated scheme
of finite type over a field of characteristic zero. Then $X$ has a
categorical resolution by  a smooth poset scheme if and only if $X$
has Du Bois singularities.
\end{theo}

The "if" direction in the theorem is essentially the definition of
Du Bois singularities (plus the work \cite{LNM1335}), and the other
direction is a consequence of the general functorial formalism which
we develop. This theorem proves the above conjecture in the case of
Du Bois singularities.

\begin{cor} Let $X$ be a reduced separated scheme of finite
type over a field of characteristic zero. Assume that $X$ has Du
Bois singularities. Then there exists a smooth DG algebra $A$ and a
categorical resolution $D(X)\to D(A),$ such that

1) $H^i(A)=0$ for $\vert i\vert >>0;$

2) $D(A)$ has a finite semi-orthogonal decomposition with summands
$D(X_i)$ where each $X_i$ is smooth and $X_1$ is a usual resolution
of $X;$

3) If $X$ is proper, then each $X_i$ is also proper.
In particular in this case the DG algebra $A$ is proper (has finite
dimensional cohomology).
\end{cor}

Theorems
\ref{degeneration-standard},\ref{Hodge-to-deRham-degener-analytic},\ref{degeneration-algebraic},\ref{posetres=DuBois},\ref{descent}
may be viewed applications of our theory of smooth projective poset
schemes to the study of Du Bois singularities. In particular,
Theorem \ref{descent} asserts that the de Rham-Du Bois complex may
be defined by means of {\it any} smooth projective poset scheme
which satisfies the descent in the classical topology.

Our poset schemes are generalizations of {\it configuration schemes}
studied in \cite{Lu1}. (A configuration scheme is a poset scheme
where all the structure morphisms $f_{\alpha \beta}$ are closed
embeddings). Although the notion of a categorical resolution is not
present explicitly in \cite{Lu1} the ideas discussed in that paper
are similar to what we do here.

\subsection{Organization of the paper}
The paper consists of two parts. In the first one we develop in
detail the theory of poset schemes and discuss their relationship
with categorical resolutions. In the second part we prove three
results on degeneration of spectral sequences for smooth projective
poset schemes (Theorems
\ref{degeneration-standard},\ref{Hodge-to-deRham-degener-analytic},\ref{degeneration-algebraic})
These results are used to prove Theorem \ref{descent}. In Theorem
\ref{posetres=DuBois} we establish a connection between Du Bois
singularities and the existence of a categorical resolution by a
smooth poset scheme.

The appendix contains some general facts on functors between derived
categories of quasi-coherent sheaves.

In \cite{Lu2} we have collected some well known general categorical
facts about cocomplete triangulated categories, existence of compact
generators, smoothness of DG algebras, existence of enough
h-injectives in derived categories of Grothendieck abelian
categories, etc. These fact are not discussed in this article and we
refer the reader to \cite{Lu2} as needed.

I want to thank Tony Pantev who first suggested a connection between
categorical resolutions by poset schemes and Du Bois singularities.
A discussion of Du Bois singularities with Karl Schwede helped me
understand the subject. Finally I thank the participants of
algebraic geometry seminar in Steklov Institute in Moscow for their
interest in this work.

\part{Categorical resolutions by poset schemes}

\section{Quasi-coherent sheaves on poset schemes}

We fix a base field $k.$ A "scheme" means a separated quasi-compact
$k$-scheme, all morphisms of schemes are assumed to be separated and
quasi-compact. All the products and tensor products are taken over
$k$ unless specified otherwise. Throughout this article a "poset"
(=a partially ordered set)  means a {\it finite} poset.

\begin{defi} Let $S=\{\alpha ,\beta ...\}$ be a poset which
we consider as a category: the set $\Hom (\alpha ,\beta )$ has a
unique element if $\alpha \geq \beta$ and is empty otherwise. Then
an {\rm  $S$-scheme}, or an {\rm $S$-poset scheme}, or a {\rm poset
scheme} is simply a functor from $S$ to the category of schemes. In
other words, a poset scheme is a collection $\cX=\{X_\alpha
,f_{\alpha \beta }\}_{\alpha \geq \beta \in S}$, where $X_\alpha$ is
a scheme and $f_{\alpha \beta }: X_\alpha \to X_\beta$ is a morphism
of schemes, such that $f_{\beta \gamma}f_{\alpha \beta}=f_{\alpha
\gamma}.$ We call $\cX$ {\rm  noetherian, regular, smooth, of finite
type, essentially of finite type, etc.} if all schemes $X_\alpha \in
\cX$ are such.
\end{defi}

\begin{defi} Let $\cX=\{X_{\alpha },f_{\alpha \beta}\}$ be a
poset scheme. A quasi-coherent sheaf on $\cX$ is a collection $F
=\{F_\alpha \in Qcoh(X_{\alpha}), \varphi _{\alpha \beta}:f_{\alpha
\beta}^*F_\beta \to F_\alpha\}$ so that the morphisms $\varphi$
satisfy the usual cocycle condition: $\varphi _{\alpha
\gamma}=\varphi _{\alpha \beta}\cdot f^*_{\alpha \beta}(\varphi
_{\beta \gamma}).$ Quasi-coherent sheaves on $\cX$ form a category
in the obvious way. We denote this category $Qcoh\cX.$
\end{defi}

\begin{lemma} \label{abelian} The category $Qcoh \cX$ is an abelian
category.
\end{lemma}

\begin{proof} Indeed, given a morphism $g:F\to G$ in $Qcoh \cX$ we
define $\Ker (g)$ and $\Coker (g)$ componentwise. Namely, put $\Ker
(g)_\alpha := \Ker (g_\alpha ),$ $\Coker (g)_\alpha :=\Coker
(g_\alpha).$ Note that $\Coker (g)$ is well defined since the
functors $f^*_{\alpha \beta}$ are right-exact.
\end{proof}

\begin{remark}\label{alternative-def} A quasi-coherent sheaf $F$
on a poset scheme $\cX =\{X_\alpha ,f_{\alpha \beta}\}$ can be
equivalently defined as a collection $F =\{F_\alpha \in
Qcoh(X_{\alpha}), \psi  _{\alpha \beta}:F_\beta \to f_{\alpha \beta
*}F_\alpha\},$ so that the morphisms $\psi$ satisfy the usual
cocycle condition: $\psi _{\alpha \gamma} =f_{\beta \gamma *}(\psi
_{\alpha \beta})\cdot \psi _{\beta \gamma}.$
\end{remark}

\begin{defi} The quasi-coherent sheaf  $\cO _{\cX}=\{ \cO
_{X_\alpha},\phi _{\alpha \beta}=\id\}$ is called the {\rm structure
sheaf of $\cX.$ } Also for each $i\geq 0$ we have the natural sheaf
$\Omega ^i_{\cX}$ - the i-th exterior power of the sheaf of Kahler
differentials $\Omega ^1_{\cX}.$ Together these sheaves form the
{\rm deRham complex} $\Omega ^\bullet _{\cX}$ (as usual the
differential in $\Omega ^\bullet _{\cX}$ is not $\cO _{\cX}$-linear;
it is a differential operator of order 1).
\end{defi}

\subsection{Operations with quasi-coherent sheaves on poset
schemes}\label{operations}

 Let $S$ be a finite poset and $\cX$ be an $S$-scheme.
 Denote for short $\cM =Qcoh\cX$ and $\cM _\alpha =QcohX_\alpha$.
For $F\in \cM$ define its support $\supp(F)=\{\alpha \in S\vert
F_{\alpha }\neq 0\}$.

Define a topology on $S$ by taking as a basis of open sets the
subsets $U_{\alpha }=\{\beta \in S\vert \beta \geq \alpha\}$.

Note that $Z_{\alpha }=\{ \gamma \in S\vert \gamma \leq \alpha \} $
is a closed subset in $S$.

Let $U\subset S$ be open and $Z=S-U$ -- the complementary closed.
Let $\cM_U$ (resp. $\cM_Z$) be the full subcategory of $\cM$
consisting of objects $F$ with support in $U$ (resp. in $Z$). For
every object $F$ in $\cM$ there is a natural short exact sequence
$$0\to F_U\to F\to F_Z\to 0,$$
where $F_U\in \cM_U$, $F_Z\in \cM_Z$. Indeed, take
$$(F_U)_{\alpha }=\begin{cases}
F_{\alpha}, &\ \text{if $\alpha\in U$}, \\
0,          &\ \text{if $\alpha\in Z$}.
\end{cases}
$$
$$(F_Z)_{\alpha }=\begin{cases}
F_{\alpha }, &\ \text{if $\alpha\in Z$}, \\
0,           &\ \text{if $\alpha\in U$}.
\end{cases}
$$
We may consider $U$ (resp.$Z$) as a subcategory of $S$ and restrict
the poset scheme $\cX$ to $U$ (resp. to $Z$). Denote these
restrictions by $\cX (U)$ and $\cX (Z)$ and the corresponding
categories by $\cM (U)$ and $\cM(Z)$ respectively.

Denote by $j:U\hookrightarrow S$ and $i:Z\hookrightarrow S$ the
inclusions. We get the obvious restriction functors
$$j^*=j^!:\cM\to\cM(U),\quad i^*:\cM \to \cM(Z).$$
Clearly these functors are exact. The functor $j^*$ has an exact
left adjoint $j_!:\cM(U)\to \cM$ (``extension by zero''). Its image
is the subcategory $\cM_U$. The functor $i^*$ has an exact right
adjoint $i_*=i_!:\cM(Z)\to \cM$ (also ``extension by zero''). Its
image is the subcategory $\cM_Z$. It follows that $j^*$ and $i_*$
preserve injectives (as right adjoints to exact functors). We have
$j^*j_!=Id$, $i^*i_*=Id$.

Note that the short exact sequence above is just
$$0\to j_!j^*F\to F\to i_*i^*F\to 0,$$
where the two middle arrows are the adjunction maps.

The functor $i_*$ also has a left-exact right adjoint functor $i^!$.
Namely $i^!F$ is the largest subobject of $F$ which is supported on
$Z$.

For $\alpha \in S$ denote by $j_{\alpha }:\{\alpha \}\hookrightarrow
S$ the inclusion. The inverse image functor $j_{\alpha }^*:\cM \to
\cM _{\alpha } ,$ $F\mapsto F_{\alpha }$ has a right-exact left
adjoint $j_{\alpha +}$ defined as follows
$$(j_{\alpha +}P)_{\beta }=\begin{cases}
f^* _{\beta \alpha}P, &\ \text{if $\beta \geq \alpha$}, \\
0, &\ \text{otherwise}.
\end{cases}
$$
Thus for $P\in \cM _\alpha $, $\supp j_{\alpha +}P\subset U_{\alpha
}$.

We also consider the ``extension by zero'' functor $j_{\alpha !}:\cM
_\alpha \to \cM $ defined by
$$j_{\alpha !}(P)_{\beta}=\begin{cases}
P, &\ \text{if $\alpha =\beta$}, \\
0, &\ \text{otherwise}.
\end{cases}
$$

\begin{lemma} The  functor $j_{\alpha }^*:\cM\to \cM_\alpha $ has a
right adjoint $j_{\alpha *}$. This functor $j_{\alpha *}$ is
left-exact and preserves injectives. For $P\in \cM_\alpha $ $\supp
(j_{\alpha
*}P) \subset Z_{\alpha }$.
\end{lemma}

\begin{proof}
Given $P\in \cM _\alpha $ we set
$$j_{\alpha *}(P)_{\gamma}=\begin{cases}
f_{\alpha \gamma *}(P), &\ \text{if $\gamma \leq \alpha$}, \\
0,                      &\ \text{otherwise},
\end{cases}
$$
and the structure map
$$\varphi _{\gamma \delta } :f^* _{\gamma \delta }((j_{\alpha *}P)_{\delta })
\to (j_{\alpha *}P)_{\gamma }$$ is the adjunction map
$$f^* _{\gamma \delta }f _{\alpha \gamma * }P=
f^* _{\gamma \delta }f _{\gamma \delta * }f _{\alpha \gamma *}P \to
f _{\alpha \delta *}P$$ if $\delta \leq \gamma \leq \alpha$ and
$\varphi _{\gamma \delta }=0$ otherwise.

It is clear that $j_{\alpha *}$ is left-exact and that $\supp
(j_{\alpha
*}P)\subset Z_{\alpha }$.

Let us prove that $j_{\alpha *}$ is the right adjoint to
$j_{\alpha}^*$.

Let $P\in \cM_\alpha $ and $M=\{ M_{\gamma},\varphi _{\gamma
\beta}\}\in \cM$. Given $g_{\alpha }\in \Hom (M_{\alpha },P)$ for
each $\gamma \leq \alpha $ we obtain a map $g_{\alpha}\cdot\varphi
_{\alpha \gamma }: f^* _{\alpha \gamma }M_{\gamma }\to P$ and hence
by adjunction $g_{\gamma }:M_{\gamma }\to f _{\alpha \gamma *}P=
(j_{\alpha
*}P)_{\gamma }$.
The collection $g=\{g_{\gamma }\} $ is a morphism $g:M\to j_{\alpha
*}P$. It remains to show that the constructed map
$$\Hom (M_{\alpha },P)\to \Hom (M, j_{\alpha *}P)$$
is surjective or, equivalently, that the restriction map
$$\Hom (M,j_{\alpha *}P)\to \Hom (M_{\alpha },P),\quad g\mapsto g_{\alpha }$$
is injective.

Assume that  $0\neq g\in \Hom (M,j_{\alpha *}P)$, i.e. $g_{\gamma
}\neq 0$ for some $\gamma \leq \alpha $. By definition we have the
commutative diagram
$$
\begin{CD}
f^* _{ \alpha \gamma }M_{\gamma }@ >f^*_{\alpha \gamma}(g_{\gamma
})>>
f^* _{\alpha \gamma }f _{\alpha \gamma * }P\\
@V\varphi _{\alpha \gamma }VV  @VV\epsilon_PV\\
M_{\alpha }@>g_{\alpha }>>P,
\end{CD}
$$
where $\epsilon_P$ is the adjunction morphism. Note that
$\epsilon_Pf^* _{\alpha \gamma }(g_{\gamma }):f^* _{ \alpha \gamma}
M_{\gamma }\to P$ is the morphism, which corresponds to $g_{\gamma
}: M_{\gamma }\to f _{\alpha \gamma * }P$ by the adjunction
property. Hence $\epsilon _Pf^* _{ \alpha \gamma }(g_{\gamma })\neq
0$. Therefore $g_{\alpha }\neq 0$. This shows the injectivity of the
restriction map $g\mapsto g_{\alpha }$ and proves that $j_{\alpha
*}$ is the right adjoint to $j_{\alpha }^*$. Finally, $j_{\alpha *}$
preserves injectives being the right adjoint to an exact functor.
\end{proof}

\begin{lemma}\label{Groth} The abelian category $\cM $ is a Grothendieck
category. In particular it has enough injectives and the
corresponding category of complexes $C(\cM)$ has enough h-injectives
\cite{KaSch},Thm.14.1.7.
\end{lemma}

\begin{proof} For a usual quasi-compact and quasi-separated scheme
$X $ the category $QcohX$ is known to be Grothendieck \cite{ThTr},
Appendix B. The category $\cM$ is abelian \ref{abelian} and has
arbitrary direct sums (since the "gluing" functors $f_{\alpha
\beta}^*$ preserve direct sums), so it has arbitrary colimits.
Filtered colimits are exact, because the exactness is determined
locally on each $X_\alpha.$ It remains to prove the existence of a
generator for the abelian category $\cM.$ For each $\alpha \in S$
choose a generator $M_\alpha \in Qcoh X_\alpha.$ We claim that $M:=
\oplus _\alpha (j_{\alpha +}M_\alpha)$ is a generator in $\cM.$
Indeed, let $g:F\to G$ be a morphism in $\cM,$ such that $g_*:\Hom
(M,F)\to \Hom (M,G)$ is an isomorphism. We have
$$\Hom (M,-)=\oplus _\alpha \Hom
(j_{\alpha +}M_\alpha ,-)=\oplus _{\alpha} \Hom (M_\alpha
,(-)_{\alpha}).$$ So for each $\alpha $ the map $g_{\alpha *}:\Hom
(M_{\alpha},F_\alpha)\to \Hom (M_{\alpha},G_{\alpha})$ is an
isomorphism, hence $g_\alpha$ is an isomorphism. Thus $g$ is an
isomorphism.
\end{proof}

\subsection{Summary of functors and their properties}\label{summary-1}
 For reader's
convenience we list all the
functors introduced so far together with their properties.

\noindent{\underline{Functors}:} $j^*=j^!,  j_!, i^*, i_*=i_!, i^!,
j^*_{\alpha}, j_{\alpha +}, j_{\alpha *}.$

\noindent{\underline{Exactness}:} $j^*, j_!, i^*, i_*, j^*_{\alpha}$
- exact; $ i^!, j_{\alpha *}$ - left-exact; $ j_{\alpha +}$ -
right-exact.

\noindent{\underline{Adjunction}:} $(j_!,j^*), (i^*, i_*),
(i_*,i^!), (j_{\alpha +}, j^*_{\alpha}), (j^*_{\alpha}, j_{\alpha
*})$ are adjoint pairs.

\noindent{\underline{Preserve direct sums}:} All the above functors
preserve direct sums. (The functor
 $j_{\alpha
*}$ preserves direct sums because the morphisms $f_{\alpha \beta}$ are
quasi-compact.)

\noindent{\underline{Preserve injectives}:} $j^*, i_*, i^!,
j_{\alpha
*}$ preserve injectives because they are right adjoint to exact
functors.

\noindent{\underline{Tensor product}:} The bifunctor $\otimes :\cM
\times \cM \to \cM$ is defined componentwise: $(F\otimes G)_{\alpha}
=F_{\alpha }\otimes _{\cO _{X_\alpha}}G_{\alpha}.$

\subsection{Cohomological dimension of poset schemes}

We keep the notation of Subsection \ref{operations}

\begin{prop}\label{finite-inj-res}
If the poset scheme $\cX$ is regular noetherian, then $\cM$  has
finite cohomological dimension.
\end{prop}

\begin{proof} The proposition asserts that any $F$ in $\cM$ has a
finite injective resolution. Equivalently, a finite complex in $\cM$
is quasi-isomorphic to a finite complex of injectives. We argue by
induction on the cardinality of $S,$ the case $\vert S\vert =1$ is
well known.

Let $\beta \in S$ be a biggest element. Put $U=U_{\beta }=\{\beta
\}$, $Z=S-U$. Let $j=j_{\beta}:U\hookrightarrow S$ and $i:
Z\hookrightarrow S$ be the corresponding open and closed embeddings.

Fix $F$ in $\cM;$ it suffices to find finite injective resolutions
for $j_!j^*F$ and $i_*i^*F$. Let $j^*F\to I_1$, $i^*F\to I_2$ be
such resolutions in categories $\cM(U)$ and $\cM(Z)$ respectively.
Then $i_*i^*F\to i_*I_2$ will be an injective resolution in $\cM$.
Note that $j_*I_1$ is a (finite) complex of injectives in $\cM$ and
that the cone $K$ of the natural morphism $j_*j^*F\to j_*I_1$ is
acyclic on $X_\beta$. Hence by the induction assumption $K$ is
quasi-isomorphic to $i_*J,$ where $J$ is a finite complex of
injectives in $\cM (Z)$. Therefore the object $j_*j^*F$ has a finite
injective resolution in $\cM.$

Consider the short exact sequence
$$0\to j_!j^*F\to j_*j^*F\to G\to 0.$$
Then $\supp(G)\subset Z$ and so by induction $G=i_*i^*G$ has a
finite injective resolution in $\cM$. Therefore the same is true for
$j_!j^*F$.
\end{proof}

\section{Derived categories of poset schemes} Let $S$ be a
poset, $\cX$ an $S$-scheme, $\cM=Qcoh\cX,$ $C(\cX)=C(\cM)$ - the
abelian category of complexes in $\cM,$ $Ho(\cX)=Ho(\cM),$
$D(\cX)=D(\cM)$ - its homotopy and derived category.

Let $U\stackrel{j}{\hookrightarrow}S\stackrel{i}{\hookleftarrow}Z$
be embeddings of an open $U$ and a complementary closed $Z$. The
exact functors $j^*,j_!,i^*,i_*, j^*_{\alpha }$ extend trivially to
corresponding functors between derived categories $D(\cM),$
$D(\cM(U)),$ $D(\cM(Z)),$ $D(X_{\alpha}).$ To define the derived
functors of the other functors we need  h-injective  and h-flat
objects in $C(\cM).$ (There are enough h-injectives by Lemma \ref{Groth})

\begin{defi} An object $F\in C(\cM)$ is called {\rm h-flat} if for any
acyclic complex $S\in C(\cM)$ the complex  $F\otimes S$ is acyclic.
\end{defi}

Notice that for any $\alpha \in S$ the functor $j_{\alpha
*}:C(X_{\alpha})\to C(\cX)$ preserves h-injectives. Indeed, its left
adjoint functor $j^*_{\alpha}$ preserves acyclic complexes. Denote
by $SI (\cX)\subset Ho(\cX)$ the full triangulated subcategory
classically generated by objects $j_{\alpha
*}M,$ for h-injective $M\in C(X_{\alpha }).$  We
call objects of $SI(\cX)$ {\it special h-injectives}. It is
sometimes convenient to use the following lemma.

\begin{lemma}\label{special-inj} There are enough special injectives in
$D(\cX).$
\end{lemma}

\begin{proof} Fix $F\in C(\cX)$ and let $\beta \in S$ be a biggest
element such that the complex $F_{\beta}$ is not acyclic. Choose an
h-injective resolution $\rho : F_\beta \to I$ in $D(X_\beta).$ By
adjunction it induces a morphism $\sigma :F\to j_{\beta *}I.$ By
construction the cone $C_{\sigma}$ of the morphism $\sigma$ is
acyclic on $X_\gamma$ for all $\gamma \geq \beta.$ So by induction
we may assume that there exists a special h-injective $J$ and a
quasi-isomorphism $C_{\sigma}\to J.$ So $F$ is quasi-isomorphic to
the (shifted) cone of a morphism $j_{\beta *}I \to J.$
\end{proof}

It is known that for any quasi-compact separated scheme $X$ there
are enough h-flats in $D(X)$ \cite{AlJeLi}, Proposition 1.1. Clearly, an
object $F\in C(\cX)$ is h-flat if and only if $F_{\alpha}\in
C(X_{\alpha})$ is h-flat for every $\alpha \in S.$ Let $M\in
C(X_{\alpha})$ be h-flat. Then $j_{\alpha +}M\in C(\cX)$ is also
such. Indeed, the inverse image functors $f^*_{\beta \alpha}$
preserve h-flats \cite{Sp}, Proposition 5.4. Denote by $S F(\cX)\subset
Ho(\cX)$ the full triangulated subcategory classically generated by
objects $j_{\alpha +}M,$ where $M\in C(X_{\alpha})$ is h-flat. We
call objects of $SF(\cX)$ {\it special h-flats}.

\begin{lemma} There are enough special h-flats in $D(\cX).$
\end{lemma}

\begin{proof} Similar to the proof of Lemma \ref{special-inj} but using the
adjoint pair $(j_{\alpha +},j^*_{\alpha})$ instead of $(j_{\alpha
}^*,j_{\alpha *}).$
\end{proof}

We now use h-injectives to define the right derived functors
$$\bR j_{\alpha
*}:D(X_\alpha)\to D(\cX), \ \ \ \bR i^!:D(\cX)\to D(\cX(Z)),$$
and h-flats to define the left derived functor
$$\bL j_{\alpha +}:D(X_{\alpha})\to D(\cX)$$
and the derived functor $(-)\stackrel{\bL}{\otimes }(-):D(\cX)\times
D(\cX)\to D(\cX)$ (by resolving any of the two variables).

\subsection{Summary of  functors and their properties} \

\noindent{\underline{Preserve h-flats and h-injectives}:} The
functors $j^*,j_!,i^*,i_*, j^*_{\alpha },j_{\alpha +}$ between the
categories $C(\cX),$ $C(\cX(U)),$ $C(\cX(Z)),$ $C(X_{\alpha})$
preserve h-flats. Also the functors $j^*,i_*,i^!,j_{\alpha *}$
preserve h-injective, since their left adjoint functors preserve
acyclic complexes.

\noindent{\underline{Derived functors}:}  We have defined the
following triangulated functors between the derived categories
$D(\cX),$ $D(\cX(U)),$ $D(\cX(Z)),$ $D(X_{\alpha})$:
 $j^*,j_!,i^*,i_*, \bR
i^!,j^*_{\alpha },\bL j_{\alpha +}, \bR j_{\alpha *}.$

\noindent{\underline{Preserve direct sums}:} All the above functors
except possibly $\bR i^!$ ($\bR j_{\alpha *}$ preserves direct sums
since the morphisms $f_{\alpha \beta}$ are quasi-compact and
separated [BoVdB],Cor.3.3.4).

\noindent{\underline{Adjunction}:} $(j_!,j^*), (i^*, i_*), (i_*,\bR
i^!), (j^*_{\alpha}, \bR j_{\alpha
*}), (\bL j_{\alpha +}, j^*_{\alpha}), $ are adjoint pairs. This
follows (except for the last pair) from the adjunctions in
Subsection \ref{summary-1} above and the fact that the functors
$j^*,i_*,i^!,j_{\alpha *}$ preserve h-injectives. For the last pair
we need a lemma.

\begin{lemma} \label{adjunction} $(\bL
j_{\alpha +}, j^*_{\alpha})$ is an adjoint pair.
\end{lemma}

\begin{proof} Choose $M\in D(X_{\alpha})$ and $I\in D(\cX).$ We need
to show that $\bR \Hom (\bL j_{\alpha +}M,I)=\bR \Hom
(M,j^*_{\alpha}I).$ We may assume that $M$ is h-flat and $I$ is a
special h-injective (Lemma \ref{special-inj}). Moreover, we then may
assume that $I=j_{\beta *}K,\beta \leq \alpha$ where $K\in
C(X_{\beta})$ is h-injective. Then $j^*_{\alpha}I=f_{\beta \alpha
*}K$  and so
$$\Hom (M, j^*_{\alpha}I)=\bR \Hom (M,j^*_{\alpha}I)$$
by Corollary \ref{maps-to-h-inj} in Appendix. Therefore
$$\begin{array}{rcl} \bR \Hom (\bL j_{\alpha +}M,I) & = & \Hom (\bL j_{\alpha +}M,I) \\
& = & \Hom (j_{\alpha +}M,I)\\
& = & \Hom (M,j^*_{\alpha}I)\\
& = & \bR \Hom (M,j^*_{\alpha}I).
\end{array}
$$
\end{proof}

\begin{defi} For $F\in D(\cX)$ we define the cohomology
$$H^i(\cX, F):=\bR ^i\Hom (\cO _{\cX},F).$$
\end{defi}

\subsection{Semi-orthogonal decompositions}

Recall that functors $j_!$ and $i_*$ identify categories $\cM(U)$
and $\cM(Z)$ with $\cM_U$ and $\cM_Z$ respectively. Denote by
$D_U(\cM)$ and $D_Z(\cM)$ the full subcategories of $D(\cM)$
consisting of complexes with cohomologies in $\cM_U$ and $\cM_Z$
respectively.

\begin{lemma}\label{full-faithful} The functors $i_*:D(\cM(Z))\to D(\cM)$
and $j_!:D(\cM(U))\to D(\cM)$ are fully faithful. The essential
images of these functors are the full subcategories $D_Z(\cM)$  and
$D_U(\cM)$ respectively.
\end{lemma}

\begin{proof} Given $F \in D_Z(\cM)$ (resp. $F \in
D_U(\cM)$) the adjunction map $F \to i_*i^*F$ (resp. $j_!j^*F \to
F$) is a quasiisomorphism. This shows that the functors
$i_*:D(\cM(Z))\to D_Z(\cM)$ and $j_!:D(\cM(U))\to D_U(\cM)$ are
essentially surjective. Let us prove that they are fully faithful.

Let $F,G\in D(\cM(Z))$ and assume that $G$ is h-injective. Then
$i_*G$ is also h-injective and we have
$$\bR \Hom(i_*F,i_*G) = \Hom(i_*F,i_*G) = \Hom(i^*i_*F,G) = \bR \Hom (F,G).$$

Similarly, let $F,G\in D(\cM(U))$ and choose a quasi-isomorphism
$j_!G\to I,$ where $I$ is h-injective. Then $j^*I$ is also
h-injective and quasi-isomorphic to $G.$ We have
$$\bR \Hom(j_!F,j_!G) = \Hom (j_!F,I) =\Hom(F,j^*I) = \bR \Hom(F,G).$$
\end{proof}

We immediately obtain the following corollary

\begin{cor}\label{cor-full-faithful}
The categories $D(\cM(U))$ and $D(\cM(Z))$ are naturally equivalent
to $D_U(\cM)$ and $D_Z(\cM)$ respectively.
\end{cor}

\begin{cor}
 Fix
$\alpha \in S$. Let $i:\{\alpha \}\hookrightarrow U_{\alpha }$ and
$j:U_{\alpha }\hookrightarrow S$ be the closed and the open
embeddings respectively. Then the functor
$$j_!\cdot i_*:D(X_{\alpha})\to D(\cM)$$
is fully faithful. In particular, the derived category $D(X_\alpha)$
is naturally (equivalent to) a full subcategory of $D(\cM)$.
\end{cor}

\begin{proof}
Indeed, by Lemma \ref{full-faithful} above the functors
$$i_*:D(X_\alpha )\to D(\cM (U_{\alpha }))$$
and
$$j_!:D(\cM (U_{\alpha }))\to D(\cM )$$
are fully faithful. So is their composition.
\end{proof}

Recall the following definitions from \cite{BoKa}.

\begin{defi}
Let $\cA$ be a triangulated category, $\cB\subset \cA$ -- a full
triangulated subcategory. A  right orthogonal to $\cB$ in $\cA$ is a full
subcategory $\cB^{\perp}\subset \cA$  consisting of all objects $C$
such that $\Hom(B,C[n])=0$ for all $n\in \bbZ$ and all $B\in \cB$.
\end{defi}

\begin{defi}
Let $\cA$ be a triangulated category, $\cB \subset \cA$ -- a full
triangulated subcategory. We say that $\cB$ is {\it
right-admissible} if for each $X\in \cA$ there exists an exact
triangle $B\to X\to C$ with $B\in \cB$, $C\in \cB^{\perp}$.
\end{defi}

Similarly one defines the left orthogonal to a full subcategory and
left admissible subcategories.

\begin{defi} Let $\cA$ be a triangulated category, $\cB , \cC \subset \cA$
-- two full triangulated subcategories. We say that $\cA$ has the
semi-orthogonal decomposition $\cA =\langle \cC ,\cB \rangle$ if
$\cC =\cB ^\perp$ and $\cB$ is right-admissible. More generally
given full triangulated subcategories $\cA _1,...,\cA _n\subset \cA$
we say that $\cA$ has the semi-orthogonal decomposition $\cA
=\langle \cA _1,\cA _2,...,\cA _n\rangle$ if

1) $\cA _1$ is right-admissible;

2) the right orthogonal $\cA _1^{\perp}$ is the category $\cD$ which
is the triangulated envelop of the categories $\cA _2,...,\cA _n;$

3) there is a semi-orthogonal decomposition
$\cD =\langle \cA _2,...,\cA _n\rangle.$
\end{defi}

\begin{lemma}\label{semi-orth}
Consider the full subcategories $D_U(\cM)$ and $D_Z(\cM)$  of
$D(\cM).$ Then

i)  $D_U(\cM)^{\perp}=D_Z(\cM)$,

ii) the subcategory $D_U(\cM)\subset D(\cM)$ is right-admissible.

\end{lemma}

\begin{proof}
i). Let $G\in D(\cM)$. Then $G\in D_U(\cM)^{\perp}\simeq (j_!D(\cM
(U)))^\perp $ iff
 $Gj^*$ is acyclic, i.e. $G\in D_Z(\cM)$.

ii). Given $X \in D(\cM)$ the required exact triangle is $X _U\to
X\to X _Z$.
\end{proof}

\begin{cor}\label{final-semi-orth} a) In the notation of Lemma
\ref{semi-orth} we have the semi-orthogonal decomposition
$D(\cM)=(D_Z(\cM),D_U(\cM)).$

b) Choose a linear ordering $\alpha _1,...,\alpha _n$ of elements of
$S$ which is compatible with the given partial order. Using
Corollary \ref{cor-full-faithful} identify each category
$D(X_{\alpha _i})$ as a full subcategory of $D(\cM)=D(\cX).$ Then
there is the semi-orthogonal decomposition
$$D(\cX)=\langle D(X_{\alpha _1}),...,D(X_{\alpha _n})\rangle.$$
\end{cor}

\begin{proof} a). This follows directly from the definitions and
Lemma \ref{semi-orth}. b) Follows from a) by
induction on the cardinality of the poset $S.$
\end{proof}

\section{Compact objects and perfect complexes on poset schemes}

Let us first recall the situation with the usual schemes.

\begin{defi} Let $T$ be a triangulated category.

a) An object $K\in T$ is called compact if the functor $\Hom
_T(K,-)$ commutes with direct sums. We denote by $T^c\subset T$ the
full triangulated subcategory of compact objects.

b) An object $K\in T^c$ is called a compact generator of $T$ if
$$K^\perp =\{M\in T\vert \Hom (K,M[n])=0\  \text{for all $n$}\}=0.$$
\end{defi}

\begin{defi} Let $X$ be a scheme. An object $G\in D(X)$
is called {\rm perfect} if locally it is quasi-isomorphic to a
finite complex of free $\cO _X$-modules of finite rank. We denote by
$\Perf (X)\subset D(X)$ the full triangulated subcategory of perfect
objects.
\end{defi}

\begin{theo}\label{BoVdB} \cite{BoVdB} Let $X$ be a scheme. Then

a) $\Perf (X)=D(X)^c,$

b) the category $D(X)$ has a compact generator.
\end{theo}

As a consequence of this theorem we obtain an equivalence of
categories $D(X)\simeq D(A)$ for a DG algebra $A.$ Namely, if $K\in
D(X)^c$ is a compact generator and $A=\bR \Hom (K,K),$ then the
functor
$$\bR \Hom (K, -):D(X)\to D(A)$$
is an equivalence (see for example \cite{Lu2}, Proposition 2.6).

We want to prove analogous results for poset schemes.

\begin{defi} Let $\cX=\{X_\alpha ,f _{\alpha \beta}\}$ be a
poset scheme. We call a complex $F=\{F_\alpha\}\in D(\cX)$  {\rm
perfect} if each $F_\alpha \in D(X_\alpha)$ is such. Denote by
$\Perf (\cX)\subset D(\cX)$ the full subcategory of perfect
complexes.
\end{defi}

\begin{remark} Notice that the functors $j^*, j_!, i^*, i_*,
j_{\alpha}^*, \bL j_{\alpha +}$ preserve perfect complexes.
\end{remark}

\begin{prop}\label{comp=perf} $D(\cX)^c=\Perf (\cX).$
\end{prop}

\begin{proof}
Fix a minimal element $\alpha \in S.$ Let $U=S-\{ \alpha \}$ and
denote by $j:U\hookrightarrow S$ and $j_{\alpha }:\{ \alpha \}
\hookrightarrow S$ the corresponding open and closed embeddings.

\begin{lemma} The functors $j_\alpha ^*,j_!,$ and $\bL j_{\alpha
+}$ preserve compact objects.
\end{lemma}

\begin{proof}  Indeed, their respective right adjoint
functors $\bR j_{\alpha *},j^*,j_\alpha ^*$ preserve direct sums.
\end{proof}

By Theorem \ref{BoVdB} the proposition holds if $\vert S\vert =1.$
So by induction we may assume that it holds for $X_\alpha $ and $\cX
(U).$

By Lemma \ref{full-faithful} the functor $j_!: D(\cX (U))\to D(\cX)$
is full and faithful with the essential image $D_U(\cX ).$ Let $M\in
D_U(\cX)$ be perfect. Then $j_!^{-1}M\in D(\cX (U))$ is also
perfect, hence compact by induction. Therefore $M=j_!(j_!^{-1}M)\in
D(\cX)$ is also compact. Vice versa, let $M\in D(\cX)^c\cap
D_U(\cX).$ Then $M\in D_U(\cX)^c$ because the inclusion
$D_U(\cX)\subset D(\cX)$ preserves direct sums. So $j_!^{-1}(M)\in
D(\cX (U))^c.$ By induction $j^{-1}_!(M)$ is perfect, so $M$ is also
perfect. We proved that $D(\cX)^c\cap D_U(\cX)=\Perf (\cX)\cap
D_U(\cX).$

Fix $F\in D(\cX)^c.$ Then $F_\alpha =j^*_\alpha F\in D(X_\alpha)^c$,
hence $F_\alpha $ is perfect by induction. Then $\bL j_{\alpha
+}j_\alpha ^*F$ is also compact and perfect. Hence the cone $C(g)$
of the canonical morphism $g:\bL j_{\alpha +}j^*_\alpha F\to F$ is
compact. But $C(g)\in D_U( \cX),$ so $C(g)\in \Perf (\cX).$ Thus
$F\in \Perf (\cX).$

Vice versa, let $F\in \Perf (\cX)$. Then $j^*F\in \Perf (\cX (U)),$
$j_\alpha ^*F\in \Perf (X_\alpha ).$ By induction $j^*F\in D( \cX
(U))^c$ and so $j_!j^*F\in D(\cX)^c.$ Also by induction $j^*_\alpha
F\in D(X_\alpha)^c.$ Consider the exact triangle
$$j_!j^*F\to F \to \bR j_{\alpha *}j_\alpha ^*F.$$
It suffices to show that $\bR j_{\alpha *}j^*_\alpha F$ is compact.
(Notice that $\bR j_{\alpha *}j_\alpha ^*F$ is perfect because
$\alpha$ is a minimal element.) We know that $\bL j_{\alpha
+}j_\alpha ^*F$ is perfect and compact. So the cone $C(p)$ of the
canonical morphism
$$p:\bL j_{\alpha +}j^*_\alpha F \to \bR j_{\alpha *}j^*_\alpha F$$
is perfect. Also $C(p)\in D_U(\cX).$ Hence  $C(p)\in D(\cX)^c$ and
so also $\bR j_{\alpha *}j^*_\alpha F$ is compact.
\end{proof}

\subsection{Existence of a compact generator}

\begin{lemma} \label{existence-comp-gen}
The category $D(\cX)$ has a compact generator.
\end{lemma}

\begin{proof} Choose a compact generator $E_\alpha \in D(
X_\alpha)$ for each $\alpha \in S.$ Put $E:=\oplus \bL j_{\alpha
+}E_\alpha.$ Then $E\in D(\cX)^c$, since the functor $\bL j_{\alpha
+}$ preserves compact objects. For $M\in D( \cX)$ we have by
adjunction
$$\Hom (E,M)=\bigoplus _\alpha \Hom (E_\alpha ,M_\alpha).$$
So $\Hom (E[i],M)=0$ for all $i$ implies that $M=0.$
\end{proof}

\begin{defi} A compact generator $E\in D(\cX)$ as constructed in the
proof of last lemma will be called {\rm special}.
\end{defi}

We get the following standard corollary.

\begin{cor} The category $D(\cX)$ is equivalent to $D(A)$ for a
DG algebra $A.$
\end{cor}

\begin{proof} If $E$ is a compact generator of $D(\cX)$ and $A=\bR
\Hom (E,E),$ then the functor
$$\bR \Hom (E,-):D(\cX)\to D(A)$$
is an equivalence of categories.
\end{proof}

\section{Smoothness of poset schemes}

In this section we prove the following theorem.

\begin{theo} \label{smooth=smooth} Let $k$ be a perfect field, $S$ - a
(finite) poset and $\cX$ a regular $S$-scheme essentially of finite
type. Then the derived category $D(\cX)$ is smooth.
\end{theo}

\begin{proof} For each $\alpha \in S$ choose a compact generator
$E_\alpha$ for $D(X_\alpha)$. Then by (the proof of) Lemma
\ref{existence-comp-gen}  the object
$$E:=\bigoplus_{\alpha \in S}\bL j_{\alpha +}E_{\alpha}$$
is a compact generator for $D(\cX).$ Put $A:=\bR \Hom (E,E).$ It
suffices to prove that the DG algebra $A$ is smooth.

Choose a minimal element $\delta \in S,$ and consider the poset
$S^\prime :=S-\{\delta\}.$ Let $\cX ^\prime :=\cX - X_{\delta}$ be
the corresponding $S^\prime$-scheme.

 Since $(\bL j_{\alpha +}E_{\alpha
})\vert _{X_{\delta}}=0$ for each $\alpha \neq \delta,$ we may
consider
$$E^\prime:=\bigoplus _{\alpha \in S^\prime}\bL j_{\alpha
+}E_{\alpha}$$ as a compact generator of $D(\cX ^\prime).$ Put
$A^\prime :=\bR \Hom (E^\prime ,E^\prime).$ (The quasi-isomorphism
type of $A^\prime$ is independent of where we compute this $\bR
\Hom$: in $D(\cX)$ or  $D(\cX ^\prime).$)

By \cite{Lu2}, Proposition 3.13 and the induction on $\vert S\vert$
we may assume that $A^\prime$ is smooth. Denote
$$A_{\delta}:=\bR Hom (\bL j_{\delta +}E_{\delta}, \bL j_{\delta
+}E_{\delta})\simeq \bR \Hom (E_{\delta},E_{\delta}).$$ Then
$A_{\delta}$ is also smooth for the same reason. Notice that $\bR
\Hom (\bL j_{\delta +}E_{\delta}, E^\prime)=0,$ hence $A$ is
quasi-isomorphic to the triangular DG algebra
$$\left( \begin{array}{cc}
           A ^\prime & 0\\
           {}_{A_{\delta}}N_{A^\prime} & A _{\delta}
           \end{array} \right),
$$
where $N=\bR \Hom (E^\prime ,\bL j_{\delta +}E_\delta).$ So by
\cite{Lu2}, Proposition 3.11 it suffices to show that the DG
$A^{\op}_{\delta}\otimes A^\prime$-module $N$ is perfect.

Consider the $S^\prime $-scheme $\cY=\cX^\prime \times X_\delta.$
That is $\cY$ consists of schemes $X_\alpha \times X_\delta$ for
$\alpha \in S^\prime$ and morphisms $f_{\alpha \beta}\times \id :
X_\alpha \times X_\delta \to X_\beta \times X_\delta.$ We denote the
inclusion $X_\alpha \times X_\delta \to \cY$ by $j_{(\alpha
,\delta)}.$

Let $E_{\delta }^*:=\bR \Hom (E_{\delta},\cO _{X_{\delta}})$ be the
dual compact generator of $D(X_{\delta}).$ Then $\bR \Hom (E_{\delta
}^* ,E_{\delta }^*)\simeq A_{\delta}^{\op}$ \cite{Lu2}, Lemma 3.15.
For each $\alpha \in S^\prime$ $E_\alpha \boxtimes E^*_\delta$ is a
compact generator of $D(X_\alpha \times X_\delta)$ \cite{BoVdB},
Lemma 3.4.1. Thus
$$\tilde{E}:=\bigoplus_{\alpha \in S^\prime}\bL j_{(\alpha ,\delta)
+}(E_\alpha \boxtimes E^*_{\delta})$$ is a special compact generator
for $D(\cY).$

\begin{lemma} There is a natural quasi-isomorphism of DG algebras
$$\bR \Hom (\tilde{E},\tilde{E})\simeq A_\delta ^{\op}\otimes
A^\prime.$$
\end{lemma}

\begin{proof} We have
$$\begin{array}{rcl}
\bR \Hom (\tilde{E},\tilde{E}) & \simeq &  \bigoplus _{\alpha \geq
\beta}\bR \Hom (\bL j_{(\alpha ,\delta)+}(E_\alpha \boxtimes
E_{\delta }^*),
\bL j_{(\beta ,\delta)+}(E_\beta \boxtimes E_{\delta }^*))\\
 & \simeq & \bigoplus _{\alpha \geq
\beta}\bR \Hom (E_\alpha \boxtimes E_{\delta }^*, \bL (f_{\alpha
\beta}\times \id)^*(E_\beta \boxtimes E_{\delta }^*))\\
 & \simeq & \bigoplus _{\alpha \geq
\beta}\bR \Hom (E_\alpha \boxtimes E_{\delta }^*, \bL f_{\alpha
\beta}^*E_\beta \boxtimes E_{\delta }^*).
\end{array}
$$
Now by \cite{Lu2}, Proposition 6.20
$$\begin{array}{cl}
 & \bR \Hom (E_\alpha \boxtimes E_{\delta }^*, \bL f_{\alpha
\beta}^*E_\beta \boxtimes E_{\delta }^*)\\
\simeq & \bR \Hom (E_\alpha ,\bL f_{\alpha \beta }^*E_\beta)\otimes
\bR \Hom (E_\delta ^*,E_\delta ^*)\\
\simeq &  \bR \Hom (E_\alpha ,\bL f_{\alpha \beta }^*E_\beta)\otimes
A_\delta ^{\op}.
\end{array}
$$
Similarly,
$$\begin{array}{rcl}
\bR \Hom (E^\prime ,E^\prime) & \simeq & \bigoplus _{\alpha \geq
\beta}
\bR \Hom (\bL j_{\alpha +}E_\alpha ,\bL j_{\beta +}E_\beta)\\
 & \simeq & \bigoplus_{\alpha \geq \beta} \bR \Hom (E_\alpha ,\bL
 f_{\alpha \beta}^*E_\beta).
 \end{array}
$$
This proves the lemma.
\end{proof}

It follows that the functor
$$\Psi _{\tilde{E}}(-):=\bR \Hom (\tilde{E},-):D(\cY)\to D(A_\delta
^{\op}\otimes A^\prime)$$ is an equivalence of categories.

For each $\alpha \in S^\prime,$ such that $\alpha >\delta$ denote by
$\Gamma (\alpha ,\delta)\subset X_\alpha \times X_\delta$ the graph
of the map $f_{\alpha ,\delta}:X_{\alpha}\to X_{\delta}.$  Define
the coherent sheaf $F$ on $\cY$ as follows. For $\alpha \in
S^\prime$ such that $\alpha >\delta$ put $F_{\alpha}:=\cO _{\Gamma
(\alpha \delta)}\in coh (X_\alpha \times X_\delta).$ If $\delta
\nless \alpha,$ then put $F_{\alpha}=0.$ The structure morphism
$\phi _{\alpha \beta}:f^*_{\alpha \beta}F_\beta \to F_\alpha$ is the
canonical isomorphism.

\begin{lemma} We have $\Psi _{\tilde{E}}(F)\simeq N.$
\end{lemma}

\begin{proof} By definition
$$\begin{array}{rcl}
N & = & \bR \Hom _{\cX}(E^\prime ,\bL j_{\delta +}E_{\delta})\\
  & = & \bigoplus _{\alpha \in S^\prime}
  \bR \Hom _{\cX}(\bL j_{\alpha +}E_{\alpha} ,\bL j_{\delta +}E_{\delta})\\
  & = & \bigoplus _{\alpha \in S^\prime}
  \bR \Hom _{X_{\alpha}}(E_{\alpha} ,\bL f^*_{\alpha \delta}E_{\delta})
  \end{array}
  $$
On the other hand
$$\begin{array}{rcl}
\bR \Hom _{\cY}(\tilde{E} ,F) & = & \bigoplus _{\alpha \in S^\prime}
\bR \Hom _{\cY}(\bL j_{(\alpha,\delta) +}(E_{\alpha}
\boxtimes E_{\delta}^*) ,F)\\
 & = & \bigoplus _{\alpha \in S^\prime} \bR
\Hom _{X_{\alpha} \times X_{\delta}}(E_{\alpha} \boxtimes
E_{\delta}^* ,\cO _{\Gamma (\alpha \delta)})\\
\end{array}$$
Let us analyze one summand in the last sum. Denote by $E_{\alpha}
\stackrel{p_{\alpha}}{\leftarrow} E_{\alpha}\times E_{\delta}
\stackrel{p_{\delta}}{\rightarrow} E_{\delta}$ and by $\gamma
:\Gamma (\alpha \delta)\to X_{\delta}$ the obvious projections.
$$\begin{array}{rl}
 & \bR
\Hom (E_{\alpha} \boxtimes E_{\delta}^* ,\cO _{\Gamma (\alpha
\delta)}) \\ = & \bR \Hom (p_{\alpha}^*E_{\alpha} \otimes
p^*_{\delta} \bR \cH om (E_{\delta}, \cO _{X_{\delta}} ),\cO
_{\Gamma (\alpha
\delta)}) \\
= &  \bR \Hom (p_{\alpha}^*E_{\alpha}, \bR \cH om (p^*_{\delta} \bR
\cH om (E_{\delta}, \cO _{X_{\delta}} ),\cO _{\Gamma (\alpha
\delta)}))\\
 = & \bR \Hom (p_{\alpha}^*E_{\alpha}, \bR \cH om _{\Gamma (\alpha
\delta)} (\bL \gamma ^*_{\delta} \bR \cH om (E_{\delta}, \cO
_{X_{\delta}} ),\cO _{\Gamma (\alpha \delta)}))\\
= &  \bR \Hom (p_{\alpha}^*E_{\alpha}, \bR \cH om _{\Gamma (\alpha
\delta)} ( \bR \cH om _{\Gamma (\alpha \delta)}(\bL \gamma
^*E_{\delta},
\cO _{\Gamma (\alpha \delta)} ),\cO _{\Gamma (\alpha \delta)}))\\
 = & \bR \Hom (p_{\alpha
}^*E_{\alpha}, \bL \gamma ^*E_{\delta})\\
 = & \bR \Hom (E_{\alpha},\bL f^*_{\alpha
\delta}E_{\delta}).
\end{array}
$$
This proves the lemma.
\end{proof}

Since the poset scheme $\cY$ is regular the object $F\in D(\cY)$ is
compact by Proposition \ref{comp=perf}. Hence $N\simeq \Psi
_{\tilde{E}} (F)\in D(A^{\op}_{\delta}\otimes A^\prime)$ is also
compact, i.e. is perfect. This proves Theorem \ref{smooth=smooth}.
\end{proof}

\section{Direct and inverse image functors for morphisms of poset schemes}

Let $S,$ $S^\prime$ be  posets and $\tau :S\to S^\prime$ be an order
preserving map. Let $\cX=\{X_\alpha ,f_{\alpha \beta}\}$ (resp. $\cX
^\prime=\{X^\prime _{\alpha ^\prime},f_{\alpha ^\prime \beta
^\prime}\}$) be an $S$-scheme (resp. an $S^\prime$-scheme).

\begin{defi}\label{tau-morphism}
A $\tau$-morphism $\cF:\cX \to \cX ^\prime $ is a collection of
morphisms $\{\cF_{\alpha}:X_\alpha \to X^\prime _{\tau (\alpha)}\}
_{\alpha \in S}$ such that for each $\alpha \geq \beta $ the
following diagram commutes
$$
\begin{array}{ccc}
X_\alpha & \stackrel{f_{\alpha \beta}}{\to} & X_{\beta}\\
\downarrow \cF_{\alpha} & & \downarrow \cF_{\beta} \\
X^\prime _{\tau (\alpha)} & \stackrel{f^\prime _{\tau (\alpha) \tau
(\beta )}}{\to} & X^\prime _{\tau (\beta)}
\end{array}
$$
\end{defi}

Let $\cF:\cX \to \cX ^\prime$ be a $\tau$-morphism and $G\in Qcoh
_{\cX ^\prime}.$ We define $\cF ^*G\in Qcoh \cX$ as follows. For
$\alpha \in S$ put $(\cF ^*G)_{\alpha}=\cF _{\alpha}^*G_{\tau
(\alpha)}$ and define the structure morphism $\phi _{\alpha \beta}:
f^*_{\alpha \beta}\cF _\beta ^*G_{\tau (\beta)}\to \cF
^*_{\alpha}G_{\tau(\alpha )}$ as $\cF ^*_{\alpha}\phi ^\prime
_{\tau(\alpha )\tau (\beta)},$ where $\phi ^\prime $ is the
structure morphism for $G.$ This defines a functor $\cF ^*:Qcoh \cX
^\prime \to Qcoh \cX .$ We also consider its left derived functor
$\bL \cF ^*:D(\cX ^\prime)\to D(\cX)$ which is defined using the
h-flats.

Notice that the functor $\cF^*$ preserves h-flats.

\begin{example} \label{pullback-differentials}
We have $\cF ^*\cO _{\cX ^\prime}= \bL \cF ^*\cO _{\cX ^\prime}=\cO
_{\cX}.$ Hence, for any $G\in D(\cX^\prime)$ we obtain the map $\cF
^*:H^\bullet (\cX ^\prime ,G) \to H^\bullet (\cX, \bL\cF ^*G);$ in
particular we get the map $\cF ^*:H^\bullet (\cX ^\prime ,\cO _{\cX
^\prime}) \to H^\bullet (\cX, \cO _{\cX}).$ Also the usual morphism
$\cF ^*\Omega ^i_{\cX ^\prime}\to \Omega ^i_{\cX}$ induces the map
$$H^\bullet (\cX ^\prime ,\Omega ^i_{\cX ^\prime})\to H^\bullet (\cX
, \Omega ^i_{\cX}).$$
\end{example}

Given another morphism of poset schemes $\cF ^\prime :\cX ^\prime
\to \cX ^{\prime \prime}$ there are natural isomorphisms of functors
$\cF ^*\cF ^{\prime *} \simeq (\cF ^\prime \cF )^*.$ Since the
functor $\cF ^{\prime *}$ preserves h-flats we also have an
isomorphism $\bL \cF ^*\cdot \bL \cF ^{\prime
*} \simeq \bL (\cF ^\prime \cF )^*.$

The functor $\cF ^*$ has the right adjoint functor $\cF _*$ which we
now describe.

We will use Remark \ref{alternative-def}

For  $\alpha ^\prime \in S^\prime$ we put $\tau ^{-1}(\geq \alpha
^\prime):=\{ \gamma \in S\vert \tau (\gamma )\geq \alpha ^\prime\}.$
Fix $F\in Qcoh \cX.$ If $\gamma \in \tau ^{-1}(\geq \alpha
^\prime),$ then $f^\prime _{\tau (\gamma)\alpha ^\prime *}(\cF
_{\gamma
*}F_{\gamma })\in Qcoh X^\prime _{\alpha ^\prime}.$ If $\delta \geq
\gamma, $ then the structure morphism $\psi _{\delta \gamma
}:F_{\gamma }\to f_{\delta \gamma *}F_{\delta}$ induces the morphism
$$f^\prime _{\tau (\gamma)\alpha ^\prime *}(\cF
_{\gamma *}F_{\gamma })\to f^\prime _{\tau (\delta)\alpha ^\prime
*}(\cF _{\delta *}F_{\delta }).$$
We define
$$(\cF _*F)_{\alpha ^\prime}=\lim _{\stackrel{\longleftarrow}{\gamma
\in \tau ^{-1}(\geq \alpha ^\prime)}}f^\prime _{\tau (\gamma)\alpha
^\prime *}(\cF _{\gamma *}F_{\gamma }).$$ If $\alpha ^\prime \geq
\beta ^\prime$ there is a natural morphism $\psi ^\prime _{\alpha
^\prime \beta ^\prime }:f^\prime _{\alpha ^\prime \beta ^\prime
*}(\cF _*F)_{\alpha ^\prime}\to (\cF _*F)_{\beta ^\prime}.$ Thus
$\cF _*F\in Qcoh \cX ^\prime$ and we get a functor $\cF _*:Qcoh \cX
\to Qcoh \cX ^\prime.$ We define its right derived functor $\bR \cF
_*:D(\cX )\to D(\cX ^\prime)$ using the h-injectives. The pairs of
functors $(\cF ^*,\cF _*)$ and $(\bL \cF ^*,\bR \cF _*)$ and
adjoint.

Given another morphism of poset schemes $\cF ^\prime :\cX ^\prime
\to \cX ^{\prime \prime}$ there are natural isomorphisms of functors
$\cF ^{\prime }_*\cF _* \simeq (\cF ^\prime \cF )_*.$ Although the
functor $\cF _*$ may not preserve h-injectives we still have a
natural isomorphism of functors $\bR \cF ^{\prime }_*\cdot \bR \cF
_* \simeq \bR (\cF ^\prime \cF )_*$ (this follows by adjunction from
the isomorphism $\bL \cF ^*\cdot \bL \cF ^{\prime
*} \simeq \bL (\cF ^\prime \cF )^*$).

The direct image functor may be computed fiberwise in case $\tau$ is
the projection of a product poset on one of the factors. Namely we
have the following lemma.

\begin{lemma}\label{direct-image-fiber} Assume that $T$ is a poset,
$S=S^\prime \times T$ is the product poset and   $\tau :S\to
S^\prime$ is the projection. Then in the above notation for any
$\alpha ^\prime \in S^\prime$ we have
$$(\cF _*F)_{\alpha ^\prime}=\lim _{\stackrel{\longleftarrow}{\gamma
\in \tau ^{-1}(\alpha ^\prime)}}\cF _{\gamma
*}F_{\gamma }.$$
\end{lemma}

\begin{proof} This is clear.
\end{proof}

\begin{example} Let $S^\prime$ consist of a single element $\alpha
^\prime$ and $X_{\alpha ^\prime}=pt.$ Then for $F\in D(\cX)$
$$\bR ^i\cF _*F=H^i(\cX ,F).$$
\end{example}

\section{Categorical resolutions by smooth poset schemes}

Let $S$ be a (finite) poset and let $\cX$ be a smooth ($S-$)poset
scheme (so that the category $D(\cX)$ is smooth by Theorem
\ref{smooth=smooth}). Let $Y$ be a scheme (which can be considered
as a poset scheme) and $\pi :\cX \to Y$ be a morphism of poset
schemes (i.e. a $\tau$-morphism for $\tau :S\to pt,$ in the
terminology of the previous section).

\begin{defi} \label{def-cat-res-by-poset-schemes} We call the morphism $\pi$ a {\rm categorical
resolution}  of $Y$ if the functor $\bL \pi ^*:D(Y)\to D(\cX)$ is a
categorical resolution, i.e. its restriction $\bL \pi ^*:\Perf
(Y)\to \Perf (\cX)$ is full and faithful.
\end{defi}

We can localize the morphism $\pi$ over $Y$ in the obvious way.
Namely, given an open subset $W\subset Y$ we denote by $\cX _W$ the
poset scheme which is the unverse image of $W$ under $\pi.$ Let $\pi
_W:\cX _W\to W$ be the corresponding morphism. If $W$ is affine
$W=Spec B,$ then the $B$-module $\bR ^j(\pi _W)_*\cO _{\cX _W}$ is
isomorphic to $H^j(\cX _W,\cO _{\cX _W}).$

\begin{prop} \label{criterion-cat-res} Let $\cX$ be a smooth poset scheme, Y be a scheme and $\pi
:\cX \to Y$ be a morphism. The following statements are equivalent.

1) $\pi$ is a categorical resolution;

2) the adjunction morphism $\cO _Y\to \bR \pi _*\cO _{\cX}$ is a
quasi-isomorphism;

3) for each affine open set $W\subset Y$ the map $H^0(W,\cO _W)\to
H^0(\cX _W,\cO_{\cX _W})$ is an isomorphism and $H^j(\cX _W,\cO_{\cX
_W})=0$ for $j>0.$
\end{prop}

\begin{proof} The equivalence 2)$\Leftrightarrow$3) is clear. It
remains to prove the equivalence 1)$\Leftrightarrow$2).

\begin{lemma}\label{adjunction-categories} Let $\cC,\cD$ be categories, $F:\cC \to \cD$ a
functor and $G:\cD \to \cC$ its right adjoint functor. Fix an object
$B\in \cC$. Then the following assertions are equivalent

a) For any object $A\in \cC$ the map $F:\Hom (A,B)\to \Hom
(F(A),F(B))$ is an isomorphism;

b) The adjunction morphism $I_B:B\to GF(B)$ is an isomorphism.
\end{lemma}

\begin{proof} The composition of the map $\Hom
(A,B)\stackrel{F}{\to}\Hom (F(A),F(B))$ with the canonical
isomorphism $\Hom (F(A),F(B))\simeq \Hom (A,GF(B))$ is equal to the
map $(I_B)_*:\Hom (A,B)\to \Hom (A, GF(B)).$
\end{proof}

Now we can prove the equivalence 1)$\Leftrightarrow$2).

Since the functor $\bL \pi ^*:D(Y)\to D(\cX)$ preserves direct sums
and perfect complexes (i.e. compact objects) it is easy to see that
it is full and faithful if and only if its restriction to the
subcategory $\Perf (Y)$ is such. (Use the fact that $D(Y)$ is the
smallest triangulated subcategory of $D(Y)$ which contains $\Perf
(Y)$ and is closed under direct sums.) Hence by Lemma
\ref{adjunction-categories} the functor $\bL \pi ^*:\Perf (Y)\to
\Perf (\cX)$ is full and faithful if and only if for every $K\in
\Perf (Y)$ the adjunction map $K\to \bR \pi _*\bL \pi ^*K$ is an
isomorphism. But the last assertion is local on $Y,$ and locally $K$
is isomorphic to a finite direct sum of shifted copies of the
structure sheaf.
\end{proof}

We give examples of categorical resolutions by smooth poset schemes
in Section \ref{examples} below.

\section{How to compute in $D(\cX)$}\label{how-to-compute}

The restriction of an h-injective object $I\in D(\cX)$ to $X_\alpha
\in \cX$ may not be h-injective.

\begin{example} $\cX =\{\pt \to \bbA^1\}$
and $I=j_*(k),$ where $j$ is the inclusion of the point $\pt$ in
$\cX.$ Then the object $I\in Qcoh\cX$ is injective, hence
h-injective as an object in $D(\cX),$ but its restriction to $\bbA
^1$ is not.
\end{example}

Nevertheless if $I\in D(\cX)$ is h-injective, then the object
$I_\alpha \in D(X_{\alpha})$ can be used to compute $\bR \Hom
(M,-),$ if $M\in D(X_{\alpha})$ is h-flat.

\begin{lemma} \label{poset-maps-to-h-inj} Let $I\in D( \cX)$ be h-injective. Fix $\alpha \in
S$ and let $M\in D(X_{\alpha})$ be h-flat. Then the complex $\Hom
(M,I_{\alpha})$ is quasi-isomorphic to $\bR \Hom (M,I_{\alpha}).$
\end{lemma}

\begin{proof} A proof of this lemma is contained in the proof of
Lemma \ref{adjunction} above.
\end{proof}

\begin{lemma} (a) Fix $\alpha \in S$ and  let $F\in D(\cX)$
be such that $F=j_{\alpha +}F_{\alpha}$ for an h-flat $F_{\alpha}\in
D(X_{\alpha}).$ Then for any $G\in D(\cX)$ we have
$$\Hom _{D(\cX)}(F,G)=\Hom
_{D(X_{\alpha})}(F_{\alpha},G_{\alpha}).$$

(b) Suppose that $\alpha \in S$ is the unique minimal element of
$S,$ i.e. $S=U_{\alpha}$ (Subsection \ref{operations}). Then for any
$G\in D(\cX)$
$$H^\bullet(\cX ,G)=H^\bullet (X_{\alpha},G_{\alpha}).$$
\end{lemma}

\begin{proof} By Lemma \ref{adjunction} the functors $(\bL j_{\alpha
+},j_{\alpha}^*)$ are adjoint, which implies (a). Now (b) follows
because $\cO _{\cX}=j_{\alpha +}\cO _{X_{\alpha}}.$
\end{proof}

The next proposition generalizes the last lemma.

\begin{prop}\label{prop-comput}
Suppose that a complex $F\in C(\cX)$ has a resolution (in $C(\cX)$)
\begin{equation} \label{resolution} 0\to K_{n}\to ...\to K_1 \to
K_0\to F\to 0\end{equation} where for each $i,$ $K_i=\oplus
_{\alpha}j_{\alpha +}M^i_{\alpha}$ with $M_{\alpha }^i\in
C(X_{\alpha})$ being h-flat.  Let $I\in C(\cX)$ be such that for
each $\alpha \in S$ and each $i,$ $\Hom (M^i_{\alpha},
I_{\alpha})=\bR \Hom (M^i_{\alpha}, I_{\alpha})$ (for example $I$ is
 h-injective as in Lemma \ref{poset-maps-to-h-inj}).
 Then the complex $\bR \Hom (F,I)$ is
quasi-isomorphic to the total complex of the double complex
\begin{equation}\label{standard-double-complex} 0\to \Hom (K_0,I)\to
\Hom (K_1,I) \to ... \to \Hom  (K_n,I)\to 0.\end{equation}
 Moreover,
for each $i$
$$\Hom  (K_i,I)=\bigoplus _{\alpha}\Hom (\bL j_{\alpha *}M^i_{\alpha},I)=
\bigoplus _{\alpha }\Hom (M_{\alpha}^i,I_{\alpha})= \bigoplus
_{\alpha }\bR \Hom (M_{\alpha}^i,I_{\alpha}).$$
 Hence
in particular we obtain a spectral sequence which converges to $\Ext
(F,I)$ with the $E_1$-term being the sum of groups $\Ext
(M^i_{\alpha},I_{\alpha})$.
\end{prop}

\begin{proof} This follows from Lemma \ref{adjunction} and
Lemma \ref{poset-maps-to-h-inj}.
\end{proof}

The following example will be of primary interest to us.

\begin{example}\label{example} In case $F=\cO _{\cX}$ one can take a
resolution \ref{resolution} with $K_i=\oplus _{\alpha }j_{\alpha
+}\cO _{X_{\alpha}},$ i.e. $M^i_\alpha=\cO _{X_{\alpha}}.$ (The same
index $\alpha$ may appear in different $K_i$'s and it may also
appear more than once in a given $K_i.$) Given $G\in D(\cX)$ choose
its h-injective replacement I. Then the double complex
\ref{standard-double-complex}  consists of sums of spaces $\Gamma
(X_{\alpha},I_{\alpha})$ and the $E_1$-term is the sum of groups
$H^\bullet (X_{\alpha },G_\alpha).$ The differential $d_1$ between
the cohomology groups is simply the sum of the maps induced by the
structure morphisms $\phi _{\alpha \beta}:f^*_{\alpha \beta}G_\beta
\to G_\alpha.$  In particular $d_1$  preserves the degree of the
cohomology groups $H^\bullet (X_{\alpha },G_\alpha).$

In case the complex $G\in D(\cX)$ is bounded below we can use
instead of an h-injective $I$ the canonical Godement resolution
$G\to \cC ^\bullet (G),$ such that for each $\alpha $ the complex
$\cC ^\bullet (G)$ consists of flabby sheaves. Notice that the
complex $\cC ^\bullet(G)$ consists of $\cO _{\cX}$-modules which are
no longer quasi-coherent (see Section \ref{variants} below).
\end{example}

\begin{defi}\label{standard-sp-seq} We call any spectral sequence
converging to $H^\bullet (\cX ,G)$ as in the above example a {\rm
standard} one. (It is not unique because one can choose different
resolutions \ref{resolution} of $\cO _{\cX}$.)
\end{defi}

\begin{example} Assume that a poset $S$ consists of 4 elements $\{
\alpha ,\beta _1,\beta _2,\beta _3\}$ where $\alpha \ge \beta _i$
for all $i$ and no other relations. Therefore we have 4 irreducible
open subsets $U_\alpha ,U_{\beta _i}\subset S.$ If $\cX$ is an
$S$-scheme one can take for example the following resolution
\ref{resolution} of the structure sheaf $\cO _{\cX}$:
$$0\to K_1\to K_0\to \cO _{\cX} \to 0,$$
where $K_0=\oplus _ij_{\beta _i+}\cO _{X_{\beta _i}}$ and $K_1=
(j_{\alpha +}\cO _{X_\alpha})^{\oplus 2}.$ This gives a standard
spectral sequence converging to $H^\bullet (\cX ,\cO _{\cX})$ with
the $E_1$-complex
$$\bigoplus _iH^\bullet (X_{\beta _i},\cO _{X_{\beta _i}})\to
H^\bullet (X_\alpha ,\cO _{\cX _\alpha })^{\oplus 2}.$$
\end{example}

\part{Poset schemes and Du Bois singularities}

\section{Other variants of poset ringed spaces}\label{variants}

Besides poset schemes and quasi-coherent sheaves on them  we can
consider "poset" versions of other usual structures. We give some
examples which will be used later. Let $\cX$ be a poset scheme.

1) One may define an abelian category $\Mod \cO _{\cX}$ just as we
defined $Qcoh \cX$ by requiring the sheaves $F_\alpha $ to be
arbitrary $\cO _{X_{\alpha}}$-modules and not necessarily
quasi-coherent ones. Moreover we may consider the abelian category
$\Sh(\cX)$ of sheaves of abelian groups on $\cX.$ (That is we
consider each $X_\alpha$ as a ringed space with the structure sheaf
$\bbZ _{X_\alpha},$ so that the gluing is by maps $\phi ^\prime
_{\alpha \beta}:f^{-1}_{\alpha \beta}F_\beta \to F_\alpha.$) Because
of the natural morphism $f_{\alpha \beta}^{-1}F_\beta \to f_{\alpha
\beta}^*F_\beta$ each object in $\Mod \cX$ defines an object of $\Sh
(\cX).$

2) Denote by $\cX ^{\et}$ the same diagram of schemes where we
consider each $X_\alpha$ with the etale topology. Let $\Sh (\cX
^{\et})$ denote the abelian category of sheaves of abelian groups on
$\cX.$ For a prime number $l$ and $n\geq 1$ let $\Sh _{l^n} (\cX
^{\et})\subset \Sh (\cX ^{\et})$ be the full subcategory of $\bbZ
/l^n$-modules.

3) If $\cX$ is a complex poset scheme of finite type we may consider
the corresponding poset analytic space $\cX ^{\an}.$ It comes with
the structure sheaf $\cO _{\cX ^{\an}}.$  (We will be interested in
$\cX ^{\an}$ only for projective $\cX.$) Again we denote by $\Sh
(\cX ^{\an})$ the abelian category of sheaves of abelian groups on
$\cX ^{\an}.$ As in the algebraic case, a sheaf of $\cO _{\cX
^{\an}}$-modules may be considered as an element of $\Sh (\cX
^{\an}).$ In particular the analytic deRham complex $\Omega ^\bullet
_{\cX ^{\an}}$ is a complex in $\Sh (\cX ^{\an})$ which is a
resolution of the constant sheaf $\bbC _{\cX ^{\an}}.$

All the functors defined in Section \ref{operations} for
quasi-coherent sheaves exist also in the categories described in
1),2),3) above. They have all the properties listed in Subsection
\ref{summary-1}.

\begin{lemma} \label{enough-injectives}
There are enough injectives in all the above categories $\Mod \cO
_{\cX},$ $\Sh(\cX),$ $\Sh (\cX ^{\et}),$ $\Sh _{l^n} (\cX ^{\et}),$
$\Sh (\cX ^{\an}),$ etc.
\end{lemma}

\begin{proof} The proof is essentially the same as the one of
Proposition \ref{finite-inj-res}.
\end{proof}

\begin{defi} Using the above lemma we may define for each bounded
below complex $L$ of sheaves in $\Sh (\cX ^?)$ its cohomology
$$H^\bullet (\cX ^?, L)=\Ext ^\bullet (\bbZ _{\cX ^?},L)$$
\end{defi}

Let $L$ is a bounded above complex of sheaves in one of the
categories in Lemma \ref{enough-injectives}. There is a spectral
sequence converging to $H^\bullet (\cX ^?,L)$ defined similarly to
Example \ref{example}. Namely, choose a resolution
$$0\to K_n\to ...\to K_0\to \bbZ _{\cX ^?}\to 0$$
where each $K_i$ is a direct sum of objects $j_{\alpha +}\bbZ
_{X_{\alpha}},$ which are extensions by zero from irreducible open
subsets $U_{\alpha}$ of the constant sheaf $\bbZ.$ Choose also an
injective resolution $L\to I.$ Then exactly as in Section
\ref{how-to-compute} we get a spectral sequence which converges to
$H^\bullet(\cX ,L).$ The $E_0$-term consists of sums of spaces
$$\Gamma (X_{\alpha},I_{\alpha})=\Hom(j_{\alpha +}\bbZ
_{X_{\alpha}},I_{\alpha})$$ and the $E_1$-term is the sum of
cohomologies $H^\bullet (X_{\alpha},L_{\alpha}).$

Notice that instead of an injective resolution $L\to I$ we could use
the canonical flabby Godement resolution $L\to G(L).$ (Since the
Godement resolution of usual sheaves is functorial it extends to
poset sheaves in $\Sh (\cX ^?).$)

\begin{defi} \label{standard-sp-seq-general}
As in the case of quasi-coherent sheaves (Definition
\ref{standard-sp-seq}) we call the above spectral sequence
converging to $H^\bullet (\cX ^?,L)$ a {\rm standard} one.
\end{defi}

\begin{remark} Assume that $L$ is a bounded below complex in $Qcoh
\cX .$ By comparing the corresponding standard spectral sequences we
conclude that the cohomology of $L$ is the same whether we consider
$L$ as a complex over $Qcoh \cX$ or over $\Sh (\cX).$
\end{remark}

\subsection{Poset GAGA}\label{GAGA}
Let $X$ be a complex projective variety, $X^{\an}$ - the
corresponding analytic space and $\iota :X^{\an}\to X$ the canonical
morphism of locally ringed spaces. For an $\cO _X$-module $F$ we
denote by $F^{\an}=\iota ^*F$ its analytization. By adjunction we
obtain the canonical morphism of sheaves $a_F:F\to \iota _*F^{\an}.$
Let $Y$ be another complex projective variety and $f:X\to Y$ be a
morphism. The adjunction morphism $a_F$ induces a morphism of
sheaves $\theta _F:(f_*F)^{\an}\to f_*^{\an}F^{\an}.$ If $F$ is
coherent then it is known by \cite{SGAI}, Expose XII, Th. 4.2 (which
is an extension of GAGA) that this morphism $\theta _F$ induces a
quasi-isomorphism $(\bR f_*F)^{\an}\to \bR f^{\an}_*F^{\an}.$ In
particular $H(X,F)=H(X^{\an},F^{\an})$ for a coherent sheaf $F.$

Let $S$ be a poset, let $\cX$ be a complex projective $S$-scheme,
and $F\in \Mod \cO _{\cX}.$ Again we denote by $F^{\an}$ - the
analytization of $F$ - the corresponding analytic sheaf on the poset
analytic space $\cX ^{\an}.$ The poset analogue of the adjunction
map $a_F$ above induces a morphism of the standard spectral
sequences for $H^\bullet (\cX ,F)$ and $H^\bullet (\cX
^{\an},F^{\an}).$ If $F\in coh \cX$ then it follows from the above
cited result in \cite{SGAI} that the induced morphism of $E_1$-terms
is an isomorphism. In particular for a coherent $F$ we have
$H^\bullet (\cX ,F)=H^\bullet (\cX ^{\an},F^{\an}).$ Moreover for
$F\in coh \cX$ the standard spectral sequence for $H^\bullet (\cX
,F)$ degenerates at $E_r$ for $r\geq 2$ if and only if the standard
spectral sequence for $H^\bullet (\cX ^{\an} ,F^{\an})$ degenerates
at $E_r.$ All the above holds also for bounded below complexes of
coherent sheaves on $\cX.$

Let $S^\prime $ be another poset and $\tau :S\to S^\prime$ - a map
of posets. Let $\cX ^\prime$ be a complex projective
$S^\prime$-scheme and $\cF :\cX \to \cX ^\prime$ - a $\tau$-morphism
(Definition \ref{tau-morphism}). Then for $F\in \coh \cX$ there is a
natural quasi-isomorphism of complexes of sheaves on $\cX ^{\prime
\an}$
$$(\bR \cF _*F)^{\an}\stackrel{\sim}{\longrightarrow}\bR
\cF^{\an}_*F^{\an}.$$ In particular, for the deRham complex $\Omega
^\bullet _{\cX}$ we have
$$(\bR \cF _*\Omega ^\bullet _{\cX})^{\an}\simeq \bR \cF^{\an}_*
\Omega^{\bullet}_{\cX ^{\an}}.$$

\section{Degeneration of the standard spectral sequence for
$H^\bullet (\cX ^{\an},\bbC)$ when $\cX$ is a smooth projective
poset scheme}

\begin{theo} \label{degeneration-standard} Let $\cX$ be a smooth complex
projective poset scheme. Then the standard spectral sequence
converging to $H^\bullet(\cX ^{\an},\bbC )=H^\bullet(\cX ^{\an},\bbC
_{\cX ^{\an}})$ degenerates at $E_2$ ($d_2=d_3=...0$). That is the
cohomology $H^\bullet(\cX ^{\an},\bbC )$ is isomorphic to the
cohomology of the complex
\begin{equation}\label{E-one}E_1=...\to \oplus H^\bullet
(X ^{\an}_{\beta},\bbC )\stackrel{\oplus f^*_{\alpha
\beta}}{\longrightarrow} \oplus H^\bullet (X ^{\an}_{\alpha},\bbC
)\to ...\end{equation}
\end{theo}

\begin{proof} We use Weil conjectures (Deligne's theorem)
\cite{De} to prove this. We follow the strategy of \cite{BBD},Ch.6
using canonical Godement flabby resolutions as in
\cite{FK},Ch.1,Sect.11,12. The argument has three steps: first we
pass from the analytic topology to the etale one, then pass to a
poset scheme over a finite field, and finally we use purity of the
Frobenius endomorphism on the etale $l$-adic cohomology of a smooth
projective scheme.

\medskip

\noindent{\bf Step 1.} Choose a prime number $l.$ Since the fields
$\bar{\bbQ}_l$ and
 $\bbC$ are isomorphic, it suffices to prove the degeneration
of the analogous spectral sequence for the cohomology groups
$H^\bullet(\cX ^{\an},\bar{\bbQ}_{l}).$

Let $Y$ be a complex scheme. We have the natural morphism of topoi
$\iota:Y ^{\an}\to Y ^{\et}.$ This morphism induces the inverse
image functor between the corresponding categories of abelian
sheaves $\iota ^*:\Sh (X^{\et})\to \Sh (X^{\an}).$ It has the
following properties \cite{FK},Ch.1,Prop.11.4.
\begin{itemize}
\item Given a morphism of schemes $f:X\to Y$ there is a natural isomorphism
of functors $f^{\an *}\cdot \iota _Y^*=\iota _X^*\cdot f^{\et *}.$
In particular, $\iota ^*$ is an exact functor.
\item For any point $y\in Y^{\an}$ and any $F\in \Sh (Y^{\et})$ the stalks
$F_y$ and $(\iota ^*F)_y$ are naturally isomorphic.
\item For a finite ring $R$ we have $\iota ^*(R_{Y^{\et}})=R_{Y^{\an}}$ and
it induces an isomorphism $H^\bullet (Y^{\an},R)=H^\bullet
(Y^{\et},R).$
\end{itemize}

Recall that the cohomology groups $H^\bullet(Y^{\et},\bar{\bbQ}_l)$
are defined as
$$H^\bullet(Y^{\et},\bar{\bbQ}_l):=
(\lim _{\leftarrow}H^\bullet(Y^{\et},\bbZ /l^n))\otimes _{\bbZ
_l}\bar{\bbQ}_l$$

It is known that the morphism $\iota$ induces an isomorphism $\iota
^*:H^\bullet (Y ^{\et},\bar{\bbQ}_l) \to H^\bullet (Y
^{\an},\bar{\bbQ}_l).$ We want to extend this result to poset
schemes.

Namely, let $\cX ^{\et}$ denote the poset scheme $\cX$ considered in
the etale topology. Similarly to the analytic case we define the
cohomology groups $H^\bullet (\cX ^{\et},\bbZ /l^n)=\Ext^\bullet
((\bbZ/l^n)_{\cX ^{\et}},(\bbZ/l^n)_{\cX ^{\et}})$ and
$$H^\bullet(\cX^{\et},\bar{\bbQ}_l):=
(\lim _{\leftarrow}H^\bullet(\cX^{\et},\bbZ /l^n))\otimes _{\bbZ
_l}\bar{\bbQ}_l$$ Again there is an obvious standard spectral
sequence converging to $H^\bullet(\cX^{\et},\bar{\bbQ}_l).$

The morphism of topoi $\iota$ induces the corresponding morphism
$\iota :\cX ^{\an}\to \cX ^{\et}$ and the functor $\iota ^*:\Sh (\cX
^{\et})\to \Sh (\cX ^{\an}).$

\begin{lemma}\label{analytic=etale}
The morphism of topoi $\iota:\cX ^{\an}\to \cX ^{\et}$ induces an
isomorphism
$H^\bullet(\cX^{\et},\bar{\bbQ}_l)=H^\bullet(\cX^{\an},\bar{\bbQ}_l).$
More precisely, there is a natural morphism of standard spectral
sequences converging to $H^\bullet (\cX^{\et},\bar{\bbQ}_l)$ and
$H^\bullet (\cX^{\an},\bar{\bbQ}_l)$ respectively, which induces an
isomorphism of the corresponding $E_1$-complexes.
\end{lemma}

\begin{proof} For each $\alpha \in S$ and $n\in \bbZ$ denote by
$(\bbZ /l^n)_{X_\alpha}\to G _{\alpha ,n}$ the canonical Godement
flabby resolution \cite{Go},\cite{FK},pp.129-130. Then naturally $G
_n= \{G _{\alpha n}\}$ is a complex in $\Sh (\cX ^{\et}).$ Moreover,
 $G_{n+1}\otimes _{\bbZ /l^{n+1}}\bbZ /l^n=G_n.$ The cohomology
$H^\bullet (\cX ^{\et},\bbZ /l^n)$ can be computed using the
resolution $G_n.$ In particular the standard spectral sequence
converging to $H^\bullet (\cX ^{\et},\bbZ /l^n)$ is defined by the
double complex $\Gamma (G_n)$ which consists of sums of groups
$\Gamma (X_\alpha ,G_n).$ These double complexes form an inverse
system
\begin{equation}\label{inverse-system} ...\to \Gamma (G_2)\to \Gamma (G_1)
\end{equation}
and the double complex
\begin{equation}\label{double-complex}
\lim _{\leftarrow}\Gamma (G_n)\otimes _{\bbZ _l}\bar{\bbQ }_l
\end{equation}
computes the cohomology $H^\bullet (\cX ^{\et},\bar{\bbQ }_l).$
Applying the functor $\iota ^*$ to the inverse system of complexes
$\{G_n\}$ provides the desider morphism of standard spectral
sequences for $H^\bullet(\cX^{\et},\bar{\bbQ}_l)$ and $
H^\bullet(\cX^{\an},\bar{\bbQ}_l)$ respectively. This morphism
induces an isomorphism of $E_1$-terms, because
$H^\bullet(X^{\et}_{\alpha},\bar{\bbQ}_l)=
H^\bullet(X^{\an}_{\alpha},\bar{\bbQ}_l)$ for each $\alpha.$
\end{proof}

So in order to prove the theorem it suffices to show the
degeneration of the standard spectral sequence for
$H^\bullet(\cX^{\et},\bar{\bbQ}_l).$

\medskip{\bf Step 2.} For any smooth complex scheme $Y$ we can find a
discrete valuation ring $V\subset \bbC$ whose residue field is the
algebraic closure of a finite field, and a smooth morphism $Y_V\to
SpecV,$ such that $Y$ is obtained by extension of scalars from
$Y_V.$ Let $Y_s$ be the closed fiber of $Y_V.$ We obtain the diagram
of schemes
$$Y\stackrel{u}{\longrightarrow}Y_V \stackrel{i}{\longleftarrow}Y_s.$$
These morphisms induce isomorphisms
$$H^\bullet (Y^{\et},\bar{\bbQ}_l)\stackrel{u^*}{\longleftarrow}
H^\bullet (Y^{\et}_V,\bar{\bbQ}_l) \stackrel{i^*}{\longrightarrow}
H^\bullet (Y^{\et}_s,\bar{\bbQ}_l).$$

This extends to smooth poset schemes. Namely, we can find $V$ as
above and a smooth poset scheme $\cX _V$ over $SpecV,$ which gives
rise to $\cX$ by extension of scalars. Let $\cX _s$ again be the
closed fiber, which is a smooth poset scheme over $\bar{\bf F}_q.$
Consider the correspodning diagram of poset schemes
$$\cX \stackrel{u}{\longrightarrow}\cX _V \stackrel{i}{\longleftarrow}
\cX _s.$$

\begin{lemma} The morphisms $u,i$ induce isomorphisms
$$H^\bullet (\cX^{\et},\bar{\bbQ}_l)\stackrel{u^*}{\longleftarrow}
H^\bullet (\cX^{\et}_V,\bar{\bbQ}_l) \stackrel{i^*}{\longrightarrow}
H^\bullet (\cX^{\et}_s,\bar{\bbQ}_l).$$ More precisely the morphisms
$u,i$ induce morphisms of the standard spectral sequences converging
to these groups. And these morphisms induces isomorphisms of the
corresponding $E_1$-terms.
\end{lemma}

\begin{proof} The proof is very similar to the proof of Lemma
\ref{analytic=etale}. Namely one considers the Godement resolution
$G_n$ of the constant sheaf $\bbZ /l^n$ on $\cX _V$ and passes to
the inverse limit. We omit the details.
\end{proof}

So it suffices to prove the degeneration of the standard spectral
sequence for $$H^\bullet (\cX^{\et}_s,\bar{\bbQ}_l).$$

\medskip

\noindent{\bf Step 3.} The geometric Frobenius endomorphism $\Fr$
acts on the smooth poset scheme $\cX _s$ and hence on the standard
spectral sequence which converges to $H^\bullet
(\cX^{\et}_s,\bar{\bbQ}_l).$ For each $\alpha \in S$ $\Fr$ acts on
$H^n (X^{\et}_{\alpha s},\bar{\bbQ}_l)$ with eigenvalues $\theta$
such that $\vert \theta \vert =q^{n/2}$ (Weil conjectures, see
\cite{De}).

In the standard spectral sequence each differential $d_r$ for $r\geq
2$ is a map between subquotients of $H^n (X^{\et}_{\alpha
s},\bar{\bbQ}_l)$ and $H^m (X^{\et}_{\beta s},\bar{\bbQ}_l)$ for
$n>m.$ Hence $d_r=0$ for $r\geq 2.$ This completes the proof of
Theorem \ref{degeneration-standard}.
\end{proof}

\section{Degeneration of Hodge to deRham spectral sequence for
smooth projective poset schemes.}

\begin{defi}
Let $\cX$ be a smooth complex projective poset scheme. Recall that
the analytic deRham complex $\Omega ^\bullet _{\cX ^{\an}}$ is a
resolution of the constant sheaf $\bbC _{\cX ^{\an}}.$ As in the
case of a single smooth variety the "stupid" filtration $F^p\Omega
^\bullet _{\cX ^{\an}}:=\oplus _{i\geq p}\Omega ^{i}_{\cX ^{\an}}$
of this deRham complex gives rise to the {\it Hodge-to-deRham}
spectral sequence converging to $H^\bullet (\cX ^{\an},\bbC ).$
\end{defi}

The following theorem is the poset scheme analogue of the well known
degeneration of the Hodge-to-deRham spectral sequence for smooth
projective varieties. The proof uses Theorem
\ref{degeneration-standard} above.

\begin{theo} \label{Hodge-to-deRham-degener-analytic}
Let $\cX$ be a smooth complex projective poset scheme. Then the
Hodge-to-deRham spectral sequence degenerates at the $E_2$-term.
That is $d_2=d_3=...=0.$ Hence
\begin{equation}\label{Hodge-decomp}H^\bullet(\cX ^{\an},\bbC )=
\bigoplus _pH^{\bullet -p}(\cX ^{\an},\Omega ^p _{\cX ^{\an}}).
\end{equation} In particular the map $\bbC _{\cX ^{\an}}\to \cO
_{\cX ^{\an}}$ induces a surjection $H^\bullet (\cX ^{\an},\bbC )\to
H^\bullet (\cX ^{\an},\cO _{\cX ^{\an}})=H^\bullet (\cX ,\cO
_{\cX}).$

The decomposition \ref{Hodge-decomp} is (contravariant) functorial
with respect to morphisms of smooth projective poset schemes.
\end{theo}

\begin{proof} The degeneration of the Hodge-to-deRham spectral
sequence follows by dimension counting from the isomorphism
\ref{Hodge-decomp}. The last assertion of the theorem is obvious. So
it suffices to prove \ref{Hodge-decomp}. To compute the cohomology
of $\bbC _{\cX ^{\an}}$ we may use the Dolbeaut resolution $\bbC
_{\cX ^{\an}}\stackrel{\sim}{\to} \Omega ^\bullet _{\cX ^{\an}}\to
\cA _{\cX ^{\an}}^{\bullet \bullet},$ where $\cA ^{p,q}$ is the
sheaf of $C^\infty$ $(p,q)$-forms. The canonical morphism of
complexes $\Omega _{\cX ^{\an}}^\bullet \leftarrow \Omega ^{\geq
p}_{\cX ^{\an}}\to \Omega ^p_{\cX ^{\an}}$ lifts to a morphism of
the corresponding Dolbeaut resolutions $\cA ^{\bullet \bullet}
\leftarrow \cA ^{\geq p,\bullet}\to \cA ^{p,\bullet}.$ Thus we
obtain the induced morphisms of standard spectral sequences
 (Definition \ref{standard-sp-seq-general})
for $\Omega _{\cX ^{\an}}^\bullet ,\Omega ^{\geq p}_{\cX ^{\an}}$
and $\Omega ^p_{\cX ^{\an}}$ respectively.

Using the usual Hodge decomposition for each $X_\alpha \in \cX$ we
find that the $E_1$-term of the standard spectral sequence for
$\Omega _{\cX ^{\an}}^\bullet$ is the direct sum of complexes
$E_1^{(p,q)},$ where $E_1^{(p,q)}$ consists of summands
$H^{p,q}(X_\alpha ,\bbC ).$ Certainly the $E_1$ term of the standard
spectral sequence for the complex $\Omega ^{\geq p}_{\cX ^{\an}}$
(resp. $\Omega ^{\leq p}_{\cX ^{\an}},$ resp. $\Omega ^{p}_{\cX
^{\an}}$) identifies as a direct summand of this complex which
consists of summands $H^{\geq p,\bullet}(X_\alpha \bbC)$ (resp.
$H^{\leq p,\bullet}(X_{\alpha},\bbC ),$ resp.
$H^{p,\bullet}(X_{\alpha},\bbC )$). By Theorem
\ref{degeneration-standard} the standard spectral sequence for the
complex $\Omega ^\bullet _{\cX ^{\an}}$ degenerates at $E_2.$
Applying the next lemma we conclude that the standard spectral
sequences for these other complexes also degenerate at $E_2.$ Now
using the dimension count we find the isomorphism
\ref{Hodge-decomp}, which proves the theorem.
\end{proof}

\begin{lemma} Let $A\to B$ be a morphism of bounded below double
complexes. Denote by $E_r(A)$ and $E_r(B)$ the $E_r$-terms of the
corresponding spectral sequences converging to $H^\bullet (Tot(A))$
and $H^\bullet(Tot(B))$ respectively.

i) Assume that the spectral sequence for $B$ degenerates at
$E_r(B),$ i.e. $0=d_r(B)=d_{r+1}(B)=...$ and the induced map of
complexes $E_r(A)\to E_r(B)$ is injective. Then the sequence for $A$
also degenerates at $E_r.$

ii) Assume that the sequence for $A$ degenerates at $E_r$ and the
map $E_r(A)\to E_r(B)$ is surjective. Then the sequence for $B$ also
degenerates at $E_r.$
\end{lemma}

\begin{proof} This is obvious.
\end{proof}

In the proof of the last theorem we also obtained the following
result.

\begin{prop}\label{analytic-degeneration} Let $\cX$ be a smooth complex projective poset scheme.
Then the standard spectral sequences converging to the cohomology of
$\cX ^{\an}$ with coefficients respectively in $\Omega ^{\geq
p}_{\cX ^{\an}}, \Omega ^{\leq p}_{\cX ^{\an}},\Omega ^{p}_{\cX
^{\an}}$ degenerate at $E_2$-terms.
\end{prop}

Now using GAGA we derive the corresponding statements in the
algebraic category. Namely let $\cX$ be a smooth complex projective
poset scheme. We consider again the "stupid" filtration $F^p\Omega
^\bullet _{\cX}:=\oplus _{i\geq p}\Omega ^{i}_{\cX }$ of the
algebraic deRham complex. It  gives rise to the spectral sequence
converging to $H^\bullet (\cX ^{\an},\Omega ^\bullet _{\cX} ).$ We
also call it "Hodge-to-de Rham".

\begin{theo} \label{degeneration-algebraic}
Let $\cX$ be a smooth complex projective poset scheme.

a) The (algebraic) Hodge-to-de Rham spectral sequence degenerates at
the $E_2$-term. That is $d_2=d_3=...=0.$ Hence
\begin{equation}\label{algebraic-Hodge-decomp}
H^\bullet(\cX ,\Omega ^\bullet _{\cX} )= \bigoplus _pH^{\bullet
-p}(\cX ,\Omega ^p _{\cX}). \end{equation}

The decomposition \ref{algebraic-Hodge-decomp} is functorial with
respect to morphisms of poset schemes.

b) The standard spectral sequences converging to the cohomology of
$\cX $ with coefficients respectively in $\Omega ^{\geq p}_{\cX },
\Omega ^{\leq p}_{\cX },\Omega ^{p}_{\cX }$ degenerate at
$E_2$-terms.
\end{theo}

\begin{proof} a) As in the analytic case everything follows from the
isomorphism \ref{algebraic-Hodge-decomp} by dimension counting. But
this isomorphism \ref{algebraic-Hodge-decomp} follows from the
isomorphism \ref{Hodge-decomp} and Subsection \ref{GAGA}.

b) This follows from Proposition \ref{analytic-degeneration} and
Subsection \ref{GAGA}.
\end{proof}

\begin{example} Let us give a simple example of a projective poset scheme
which is not smooth and for which the standard spectral sequence
converging to $H^\bullet (\cX ,\cO _{\cX})$ does not degenerate at
$E_2.$ Namely, let $X$ be be a projective curve which is the union
of two projective lines $C_1$ and $C_2$ which intersect
transversally at 2 points $p_1$ and $p_2.$ Then $H^1(X,\cO _X)$ has
dimension 1. Now take two copies of the curve $X=X_1=X_2,$ and let
the poset scheme $\cX$ consist of $X_1,X_2,C_1,C_2,p_1,p_2$ with the
obvious maps from each of the $C$'s (resp. $p$'s) to each of the
$X$'s (resp. $C$'s). Then a standard spectral sequence converging to
$H^\bullet (\cX ,\cO _{\cX})$ has for the $E_1$-term the natural
complex
$$0\to H^\bullet(X_1)\oplus H^\bullet (X_2)\to
H^\bullet(C_1)\oplus H^\bullet (C_2)\to H^\bullet(p_1)\oplus
H^\bullet (p_2)\to 0 $$ where $H^\bullet (Y)$ denotes $H^\bullet
(Y,\cO _Y).$ Let $0\neq a\in H^1(X,\cO _X).$ Then $(a,-a)$ is a
nonzero cycle in the above complex and it is not difficult to check
that $d_2(a,-a)\neq 0.$
\end{example}

\section{Cubical hyperresolutions and Du Bois singularities}

Cubical hyperresolutions are poset schemes of a certain type. Here
we briefly recall the definition and the main properties of cubical
hyperresolutions according to \cite{LNM1335},Ex.1.

For each integer $n\geq -1$ we denote by by $\square ^+_n$ the poset
which is the product of $n+1$ copies of the poset $\{0,1\}$. Thus
for $n=-1$ the poset $\square ^+_{-1}$ consists of one element and
$\square ^+_{0}=\{0,1\}.$ Let $\square _n$ denote the complement in
$\square ^+_{n}$ of the initial object $(0...0).$ For $\alpha
=(\alpha _0...\alpha _n)\in \square ^+_n$ we put $\vert \alpha \vert
=\alpha _0+...+\alpha _n.$

\begin{defi} Let $S$ be a (finite) poset,
$\cX$ be a reduced separated $S$-scheme
of finite type, and let $\cZ $ be a reduced $\square ^+_1\times
S$-scheme. We call $\cZ $ a 2-resolution of $\cX$ if for each $\beta
\in S$ the commutative diagram
$$\begin{array}{ccc}
Z_{11\beta} & \to & Z_{01\beta}\\
\downarrow & & \downarrow f\\
Z_{10\beta } & \to & Z_{00\beta}
\end{array}
$$
has the following properties:

1) it is a cartesian square,

2) $Z_{00\beta }=X_{\beta},$

3) $Z_{01\beta}$ is smooth,

4) horizontal arrows are closed embeddings,

5) the morphism $f$ is proper,

6) $Z_{10\beta}$ contains the discriminant of $f.$  In other words
$f$ induces an isomorphism $f:Z_{01\beta }\backslash
Z_{11\beta}\stackrel{\sim}{\to} Z_{00\beta }\backslash Z_{10\beta}.$
\end{defi}

\begin{defi} Fix a poset $S$ and an integer $r\geq 1.$
Assume that for each $1\leq n \leq r$ we are given an
$\square ^+_{n}\times S$-scheme $\cX ^n$ so that the $\square
^+_{n-1}\times S$ schemes $\cX ^{n+1}_{00\bullet}$ and $\cX
^n_{1\bullet}$ are equal. We define by induction on $r$ an $\square
^+_r\times S$-scheme $\cZ =\rm{rd}(\cX ^1,\cX^2,...,\cX ^r),$ which
we call the reduction of $(\cX ^1 ,...,\cX ^r ).$ Namely, if $r=1$
we put $\cZ  =\cX ^1.$ If $r=2$ we define
$$Z _{\alpha \beta}= \left\{ \begin{array}{ll}
X^1_{0\beta}, & \text{if $\alpha =(00),$}\\
X^2_{\alpha \beta}, & \text{if $\alpha \in \square _1$}\\
\end{array} \right.
$$
for all $\beta \in \square ^+_0\times S.$ For $r>2$ we put
$$\cZ  =\rm{rd}(\rm{rd}(\cX ^1  ,...,\cX ^{r-1}),\cX ^r).$$
\end{defi}

\begin{defi} Let $S$ be a poset and $\cX$ be an $S$-scheme.
An {\rm augmented cubical hyperresolution} of $\cX$ is an $\square
^+_r\times S$-scheme $\cZ ^+$ such that
$$\cZ ^+=\rm{rd}(\cX ^1,...,\cX ^r),$$
where

1) $\cX^1$ is a 2-resolution of $\cX,$

2) for each $1\leq n \le r,$ $\cX ^{n+1}$ is a 2-resolution of $\cX
^n_{1\bullet},$ and

2) $Z_{\alpha }$ is smooth for each $\alpha \in \square _r.$

We will call the $\square _r$-scheme $\cZ =\cZ ^+ \backslash
Z_{(0,...,0)}$ a {\rm  cubical hyperresolution} of $\cX.$ It comes
with the augmentation morphism of poset schemes $\pi :\cZ \to \cX,$
which is compatible with the projection of posets $\square _r\times
S\to S.$
\end{defi}

\begin{theo} Assume that the base field $k$ has characteristic zero.
Let $S$ be a poset and $\cX $ be a separated reduced $S$-scheme of
finite type. Then there exists an augmented cubical hyperresolution
$\cZ$ of $\cX,$ such that $\dim Z_{\alpha}\leq \dim \cX -\vert
\alpha \vert +1.$
\end{theo}

\begin{prop} \label{fiber-hyperres} Let $S$ be a poset,
$\cX$ an $S$-scheme and $\cZ$ an
$\square ^+_r\times S$-scheme, which is an augmented cubical
hyperresolution of $\cX.$ Then for each $\alpha \in S$ the $\square
^+_r$-scheme $\cZ _{\bullet \alpha}$ is an augmented cubical
hyperresolution of $X_{\alpha}.$
\end{prop}

We refer the reader to \cite{LNM1335},Ex.1,Thm.2.15,Prop.2.14 for
the proof of the above theorem and proposition and also for the
study of the category of cubical hyperresolutions of $S$-schemes.

\begin{remark} Let $X$ be a reduced separated complex scheme of
finite type and let $\pi :\cZ \to X$ be a cubical hyperresolution.
Then $\bR \pi ^{\an}_*\bbC _{\cZ ^{\an}}=\bbC _{X^{\an}}.$ This
follows from \cite{LNM1335},Ex.1,Thm.6.1.
\end{remark}

\begin{defi} Let $X$ be a reduced separated scheme of finite type over a
field of characteristic zero. Choose its cubical hyperresolution
$\pi :\cZ \to X.$ We say that $X$ has Du Bois singularities ($X$ is
Du Bois, for short) if the adjunction morphism $\cO _X\to \bR \pi
_*\cO _{\cZ}$ is a quasi-isomorphism.
\end{defi}

\begin{remark} The complex $\bR \pi _*\cO _{\cZ}\in D(X)$ is independent
(up to a quasi-isomorphism) on the choice of a hyperresolution of
$X$ (\cite{LNM1335},Ex.3). So the notion of Du Bois singularities is
well defined.
\end{remark}

\begin{remark} \label{DuBois-sing} If $X$ has rational singularities (for example $X$ is
smooth), then $X$ is Du Bois. It was conjectured by Kollar \cite{Ko}
and recently proved by Kollar and Kovac \cite{KoKov} that if $X$ has
log canonical singularities, then $X$ is Du Bois.
\end{remark}

\begin{theo}\label{Kovac} Let $X$ be a reduced separated scheme of
finite type over a field of characteristic zero. Choose its
hyperresolution $\pi :\cZ \to X.$ Assume that the adjunction map
$\cO _X\to \bR \pi _*\cO _{\cZ}$ has a left inverse. Then $X$ is Du
Bois (i.e. this map is a quasi-isomorphism).
\end{theo}

\begin{proof} See \cite{Kov}.
\end{proof}

The notion of Du Bois singularities characterizes  the existence of
categorical resolutions by smooth poset schemes as is shown in the
next theorem.

\begin{theo}\label{posetres=DuBois} Let $X$ be a reduced scheme of finite type over a field
of characteristic zero. Then there exists a categorical resolution
of $X$ by a smooth poset scheme (Definition
\ref{def-cat-res-by-poset-schemes}) if and only if $X$ has Du Bois
singularities.
\end{theo}

\begin{proof} One direction is clear: if $X$ has Du Bois singularities
 and $\pi :\cZ \to X$ is its hyperresolution then by Proposition
 \ref{criterion-cat-res} $\pi$ is a
categorical resolution of $X$ by the smooth poset scheme $\cZ.$

Vice versa, assume that $S$ is a poset, $\cX$ is a smooth $S$-scheme
and $\sigma :\cX \to X$ is a categorical resolution. Consider the
augmented $S^+:=S\cup \{0\}$-scheme $\cX ^+$ defined by $\sigma$ (so
that $X_0=X$). A choice of a hyperresolution of $\pi :\cY \to \cX
^+$ induces a commutative diagram of poset schemes
$$\begin{array}{ccc}
\cY & \stackrel{\pi}{\longrightarrow} & \cX \\
\downarrow \tilde{\sigma}& & \downarrow \sigma\\
\cY _0 & \stackrel{\pi _0}{\longrightarrow} & X
\end{array}
$$
which is compatible with the diagram of projections of posets
$$\begin{array}{ccc}
\square _n\times S & \to & S\\
\downarrow & & \downarrow\\
\square _n & \to & \{0\}
\end{array}
$$
and such that $\pi _0$ (and $\pi $) are hyperresolutions
(Proposition \ref{fiber-hyperres}).

By our assumption the adjunction map $\cO _{X}\to \bR \sigma _*\cO
_{\cX}$ is an isomorphism, and we want to prove that the adjunction
morphism $\cO _{X}\to \bR (\pi _0)_*\cO _{\cY _{0}}$ is an
isomorphism. By Theorem \ref{Kovac} it suffices to prove that this
last map has a left inverse.

Since the poset scheme $\cX$ is smooth we conclude by Remark
\ref{DuBois-sing}, Proposition \ref{fiber-hyperres} and Lemma
\ref{direct-image-fiber} that the map $\cO _{\cX}\to \bR \pi _{*}\cO
_{\cY }$ is an isomorphism. Thus the adjunction map $\cO _X\to \bR
(\sigma \cdot \pi)_*\cO _{\cY}=\bR (\pi _0\cdot \tilde{\sigma})\cO
_{\cY}$ is an isomorphism. But this last map is the composition of
the adjunction maps $\cO _X\to \bR (\pi _0)_*\cO_{\cY _0}\to \bR
(\pi _0)_*\cdot\bR (\tilde{\sigma})_*\cO_{\cY}.$ Hence the map $\cO
_{X}\to \bR (\pi _0)_*\cO _{\cY _{0}}$ has a left inverse. This
proves the theorem.
\end{proof}

Cubical hyperresolutions give more: one can define the de Rham
complex of a singular algebraic variety $X.$ Namely, choose a
hyperresolution $\pi :\cZ \to X$ and define the de Rham-Du Bois
complex $\underline{\Omega}^\bullet _X:=\bR \pi _*\Omega ^\bullet
_{\cZ}.$ This complex consists of $\cO _X$-modules and has the
differential which is a differential operator of order 1. It has
coherent cohomology and is well defined (independent of the choice
of a hyperresolution) up to a quasi-isomorphism in the appropriate
derived category  \cite{LNM1335},Ex.3. There exists a canonical
morphism of filtered complexes from the usual de Rham complex
$\Omega ^\bullet _X$ to $\underline{\Omega}^\bullet _X$ which is a
quasi-isomorphism if $X$ is smooth.

If $X$ is a reduced separated complex scheme, then the analytization
$(\underline{\Omega}^\bullet _X)^{\an}=\underline{\Omega}^\bullet
_{X^{\an}}$ is a resolution of the constant sheaf $\bbC _{X^{\an}}.$

The stupid filtration of the complex $\Omega ^\bullet _{\cZ}$
induces a filtration on the de Rham-Du Bois complex and
$\underline{\Omega}^\bullet _X$ is well defined even as a filtered
complex. The associated graded pieces are
$\underline{\Omega}^i_X:=\gr ^i\underline{\Omega}^\bullet _X=\bR \pi
_*\Omega ^i_{\cZ}.$ If $X$ is proper then this filtration induces
the Hodge filtration on $H^\bullet (X^{\an},\bbC).$

We will prove in Theorem \ref{descent} below that for a reduced
complex projective scheme $X$ the filtered complex
$\underline{\Omega}^\bullet _X$ can be defined as $\bR \sigma
_*\Omega ^\bullet _{\cX},$ where $\cX$ is a smooth complex
projective poset scheme and $\sigma :\cX \to X$ is a morphism such
that $\bR \sigma ^{\an}_*\bbC _{\cX ^{\an}}=\bbC _{X^{\an}}.$

\section{Examples of categorical resolutions by smooth poset
schemes}\label{examples}

Let $Y$ be a reducible scheme with irreducible components
$Y_1,...,Y_n.$ Assume that for each $1\leq k\leq n$ and each subset
$\alpha =\{i_1,...i_k\}\subset \{1,...n\}$ the scheme
$$X_{\alpha}:=\bigcap _{j=1}^kY_{i_j}$$
is smooth. (In particular the components $Y_i$ are smooth.) Let $S$
be the poset of nonempty subsets of $\{1,...,n\}$ with the natural
partial ordering by inclusion. Let $\cX =\{ X_{\alpha }\}$ be the
corresponding smooth poset scheme with the maps $f_{\alpha
\beta}:X_{\alpha }\to X_{\beta}$ being the obvious inclusions. Let
$\pi :\cX \to Y$ be the natural morphism.

\begin{prop} The functor $\bL \pi ^*:D(Y)\to D(\cX)$ is a
categorical resolution of singularities, i.e. the functor
$$\bL \pi^*:\Perf (Y)\to \Perf (\cX)$$
is full and faithful.
\end{prop}

\begin{proof} By Proposition \ref{criterion-cat-res}
we may assume that $Y$ is affine and we
only need to prove that the map $\Ext (\cO _Y,\cO _Y) \to \Ext (\cO
_{\cX},\cO _{\cX})$ is an isomorphism.

We have $\Ext ^i(\cO _Y,\cO _Y)=0$ for $i\neq 0.$ On the other hand
we have the obvious complex in $C(\cX)$
$$C(\cO _{\cX}):=... \to \bigoplus _{\vert \alpha \vert =2}
j_{\alpha +}(\cO _{\cX})_{\alpha} \to \bigoplus _{\vert \beta \vert
=1}j_{\beta +}(\cO _{\cX})_{\beta} \to 0,$$ which is a resolution of
$\cO _{\cX}.$ Since all schemes $X_{\alpha}$ are affine we have
$\Hom (C(\cO _{\cX}), \cO_{\cX})=\bR \Hom (\cO _{\cX},\cO _{\cX})$
(Example \ref{example}). But $\Hom (C(\cO _{\cX}),\cO _{\cX})$ is
the complex
$$0\to \bigoplus _{\vert \beta \vert =1}H^0(X_\beta, \cO
_{X_{\beta}}) \to \bigoplus _{\vert \alpha \vert =2}H^0(X_\alpha,
\cO _{X_{\alpha}}) \to ...$$ which is quasi-isomorphic to $H^0(Y,
\cO_Y).$
\end{proof}

\subsection{Categorical resolution of the cone over a plane cubic}
Here we show how smooth poset schemes can be used to construct a
categorical resolution of the simplest nonrational singularity - the
cone over a smooth plain cubic.

Let $C\subset \bbP^2$ be a smooth curve of degree 3 (and genus 1)
and $Y\subset \bbP ^3$ be the projective cone over $C.$ So $Y$ is a
cubic surface with a singular point $p$ - the vertex of the cone. We
have
$$H^i(Y,\cO _Y)=\left\{ \begin{array}{ll}
k, & \text{if i=0}\\
0, & \text{otherwise.}\\
\end{array} \right.
                       $$
Let $f:X\to Y$ be the blowup of the vertex, so that $X$ is a smooth
ruled surface over the curve $C.$ Denote by
$i:E=f^{-1}(p)\hookrightarrow X$ the inclusion of the exceptional
divisor. We have
$$H^i(X,\cO _X)=\left\{ \begin{array}{ll}
k, & \text{if i=0,1}\\
0, & \text{otherwise,}\\
\end{array} \right.
                       $$
and the pullback map $i^*:H ^\bullet(X,\cO _X)\to H ^\bullet (E,\cO
_E)$ is an isomorphism.

Consider the following smooth poset scheme $\cX$
$$\begin{array}{ccc}
 E & \to  & X\\
 \downarrow &  &\\
q & &
 \end{array}
 $$
 where $q=Spec k,$ and the map $E\to X$ is the embedding $i.$ Denote
 by $\pi :\cX \to Y$ the obvious morphism which extends the blowup
 $f:X\to Y.$

\begin{prop} $\bL \pi^*:D(Y)\to D(\cX)$ is a categorical resolution
of singularities, i.e. the functor
$$\bL \pi ^*:\Perf (Y)\to \Perf (\cX)$$
is full and faithful.
\end{prop}

\begin{proof} Note that the map $\pi $ is an isomorphism away from
the point $p\in Y.$ So we may replace $Y$ by the corresponding {\it
affine} cone $Y_0$ over $C,$ $f_0:X_0\to Y_0$ is still the blowup of
the vertex and the rest is the same. Denote the corresponding poset
scheme by $\cX _0.$ Then it suffices to prove that the map
$H^\bullet (Y_0,\cO _{Y_0})\to H^\bullet (\cX _0,\cO _{\cX _0})$ is
an isomorphism. We have $H^i(Y_0,\cO_{Y_0})=0$ for $i\neq 0.$ To
compute $H (\cX _0,\cO _{\cX _0})$ we may use the spectral sequence
as in Example \ref{example}. Then the $E_1$-term is the sum of the
two complexes:
$$k\oplus \Gamma (X _0,\cO_{X_0}) \to \Gamma (E,\cO _E),\quad \text{and}
\quad H^1(X _0,\cO _{X_0}) \to H^1(E,\cO _E).$$ The second map is an
isomorphism, and the first one is surjective with the kernel $\Gamma
(Y_0,\cO _{Y_0}).$
\end{proof}

In view of Theorem \ref{posetres=DuBois} above the last example is a
special case of the following result of Du Bois
\cite{DuB},Prop.4.13.

\begin{prop} Let $W\subset \bP ^m$ be a smooth variety such that for
all $i>0$ and $n>0$ the following holds
$$H^i(W,\cO (n))=0.$$
Then the cone over $W$ has Du Bois singularities.
\end{prop}

\begin{remark} In fact, using the same construction as in the above
example of the
cone over a smooth cubic curve it is easy to see that the condition
in the last proposition is necessary for the cone over $W$ to be Du
Bois. For example if $W\subset \bbP ^2$ is a smooth curve of degree
$\geq 4,$ then the cone over $W$ is not Du Bois.
\end{remark}

Some other examples of Du Bois singularities are listed in
\cite{St}. For example if $X$ is a reduced curve, then $X$ is Du
Bois if and only if at every singular point of $X$ the branches are
smooth and their tangent directions are independent.

\section{Descent for Du Bois singularities}

\begin{theo}\label{descent} Let $X$ be a reduced complex projective scheme.
Let $\cX $ be a smooth complex projective poset scheme and $\sigma
:\cX \to X$ be a morphism such that the adjunction map $\bbC
_{X^{\an}}\to \bR \sigma ^{\an}_*\bbC _{\cX ^{\an}}$ is a
quasi-isomorphism. Consider the direct image $\bR \sigma _*\Omega
^\bullet _{\cX}.$ This complex has a filtration induced by the
stupid filtration of the de Rham complex $\Omega ^\bullet _{\cX}.$
Then there exists a natural morphism of filtered complexes
$$\tau :\underline{\Omega}^\bullet _X\to \bR \sigma _*\Omega ^\bullet _{\cX}$$
which is a quasi-isomorphism. In particular, the map
$$\gr ^i\tau :\underline{\Omega}^i _X\stackrel{\sim}{\to} \bR \sigma
_*\Omega ^i_{\cX}$$ is a quasi-isomorphism for all $i\geq 0.$
 So if $X$ has Du Bois
singularities, then $\underline{\Omega}^0_X\simeq \cO _X \simeq \bR
\sigma _*\cO _{\cX},$ i.e. the functor $\bL \sigma ^*:D(X)\to
D(\cX)$ is a categorical resolution of singularities.
\end{theo}

\begin{proof}
As in the proof of Theorem \ref{posetres=DuBois} choose a
commutative diagram
\begin{equation}\label{main-diag}\begin{array}{lcl}
\cY & \stackrel{\pi}{\longrightarrow} & \cX \\
\downarrow \tilde{\sigma}& & \downarrow \sigma\\
\cY _0 & \stackrel{\pi _0}{\longrightarrow} & X
\end{array}
\end{equation}
where $\pi _0$ is a hyperresolution and for each scheme $X_\alpha
\in \cX$ the induced morphism $\pi :\pi ^{-1}(X_\alpha)\to
X_{\alpha}$ is also a hyperresolution.

Since each $X_\alpha$ is smooth we have the quasi-isomorphism of
filtered complexes $\Omega ^\bullet_{\cX}\stackrel{\sim}{\to}\bR \pi
_*\Omega ^\bullet _{\cY}.$ It follows that $\bR \sigma _*\Omega
^\bullet _{\cX}\simeq \bR (\sigma \cdot \pi)_*\Omega ^\bullet
_{\cY}=\bR (\pi _0 \cdot \tilde{\sigma})_*\Omega ^\bullet _{\cY}.$
On the other hand by definition $\bR (\pi _0)_*\Omega ^\bullet _{\cY
_0}=\underline{\Omega}^\bullet _X.$ Hence the adjunction morphism
$\theta :\Omega ^\bullet _{\cY _0}\to \bR \tilde{\sigma}_*\Omega
^\bullet _{\cY}$ induces the desired morphism of filtered complexes
\begin{equation}\label{morphism}
\tau :\underline{\Omega}^\bullet _X = \bR (\pi _0)_*\Omega ^\bullet
_{\cY _0}\stackrel{\bR (\pi _0)_*\theta}{\longrightarrow} \bR (\pi_0
\cdot \tilde{\sigma})_*\Omega ^\bullet _{\cY}\simeq \bR \sigma
_*\Omega ^\bullet _{\cX}.
\end{equation}
We will prove that for each $i$ the map
$$\gr ^i\tau :\underline{\Omega}^i _X \to \bR \sigma
_*\Omega ^i _{\cX}$$ is a quasi-isomorphism (hence $\tau$ is a
quasi-isomorphism).

\begin{lemma} \label{isom-hypercohom} For each $i$
the morphism $\gr ^i\tau$ induces an isomorphism on the
hypercohomology
$$H^\bullet(\gr ^i\tau):H^\bullet (X,\underline{\Omega} ^i
_X)\to H^\bullet (X,\bR \sigma _*\Omega ^i _{\cX}).$$
\end{lemma}

\begin{proof} Note that the map $H^\bullet(\gr ^i\tau)$ coincides with
 the inverse image map $H^\bullet (\cY _0,\Omega ^i _{\cY _0})\to
H^\bullet (\cY ,\Omega ^i _{\cY })=H^\bullet (X,\bR \sigma _*\Omega
^i _{\cX}).$

The diagram \ref{main-diag} induces the corresponding
diagram of analytic spaces
\begin{equation}\label{main-diag-analytic}\begin{array}{lcl}
\cY^{\an} & \stackrel{\pi ^{\an}}{\longrightarrow} & \cX ^{\an} \\
\downarrow \tilde{\sigma}^{\an}& & \downarrow \sigma ^{\an}\\
\cY _0^{\an} & \stackrel{\pi _0^{\an}}{\longrightarrow} & X^{\an}
\end{array}
\end{equation}

By Subsection \ref{GAGA} it suffices to show that the corresponding
inverse image map $H^\bullet (\gr ^i\tau ^{ \an}):H^\bullet (\cY
_0^{\an},\Omega ^i _{\cY _0^{\an}})\to H^\bullet (\cY_{\an} ,\Omega
^i _{\cY^{\an}})$ is an isomorphism.

 Since $\pi _0 $ and $\pi $ are cubical hyperresolutions we have
 $\bR (\pi _0^{\an})_*\bbC _{\cY
_0^{\an}}=\bbC _{X ^{\an}}$ and $\bR (\pi ^{\an})_*\bbC _{\cY
^{\an}}=\bbC _{\cX ^{\an}}.$ Thus by our assumption $\bR (\sigma
^{\an}\cdot \pi ^{\an})_*\bbC _{\cY ^{\an}}= \bbC _{X^{\an}}.$ As in
the case of the sheaves $\Omega ^i$  we obtain a natural morphism
$$\tau ^c:\bR (\pi _0^{\an})_*\bbC _{\cY _0^{\an}}\to \bR \sigma
^{\an}_*\bbC _{\cX ^{\an}}$$ which is a quasi-isomorphism (both
sides are quasi-isomorphic to $\bbC _{X^{\an}}$). Hence the map
$$H^\bullet(\cY _0^{\an},\bbC ) \stackrel{H^\bullet (\tau ^c)}
{\longrightarrow}  H^\bullet (\cX ^{\an},\bbC )=H^\bullet (\cY
^{\an},\bbC )$$ is an isomorphism.

By Theorem \ref{Hodge-to-deRham-degener-analytic}
\begin{equation}H^\bullet(\cY _0 ^{\an},\bbC )=
\bigoplus _iH^{\bullet -i}(\cY _0 ^{\an},\Omega ^i _{\cY _0^{\an}}).
\end{equation}
 and similarly for
$\cY.$ The map $H^\bullet (\tau ^c)$ respects this decomposition and
its restriction to the i-th summand is the map $H^\bullet (\gr
^i\tau ^{\an}).$ It follows that $H^\bullet (\gr ^i\tau ^{\an})$ is
also an isomorphism. This proves the lemma.
\end{proof}

\begin{lemma} \label{apply-serre} Let $Y$ be a complex projective scheme with an ample
line bundle $L.$ Let $u:K_1\to K_2$ be a morphism of complexes in
$D^b(cohY).$ Assume that for all $n>>0$ the map $u$ induces an
isomorphism of the hypercohomology
$$H^\bullet (Y,K_1\otimes L^n)\stackrel{\sim}{\longrightarrow}
H^\bullet (Y,K_2\otimes L^n).$$ Then $u$ is a quasi-isomorphism.
\end{lemma}

\begin{proof} See Lemma 3.4 in \cite{LNM1335} (p.139).
\end{proof}

We will prove that the morphism $\gr \tau ^i$ satisfies the
assumptions of Lemma \ref{apply-serre}, which will prove the
theorem.

\begin{prop} Let $L$ be an ample line bundle on $X.$ Then for any
$n\geq 1$ the map $\gr ^i\tau \otimes L^n:\gr ^i\tau
:\underline{\Omega}^i _X\otimes L^n \to (\bR \sigma _*\Omega ^i
_{\cX})\otimes L^n$ induces an isomorphism on hypercohomology
$$H^\bullet (X, \underline{\Omega}^i _X\otimes L^n) \to H^\bullet (X,
(\bR \sigma _*\Omega ^i _{\cX})\otimes L^n).$$
\end{prop}

\begin{proof} We prove the proposition by induction on the dimension
of $X.$ If $\dim X=0,$ then the statement is equivalent to Lemma
\ref{isom-hypercohom}.

We denote by $L$ also the pullbacks of $L$ to the smooth poset
schemes $\cX$ and $\cY _0.$ By the projection formula it suffices to
prove that the natural map
$$H^\bullet (X, \bR \pi _{0 *}(\Omega ^i _{\cY _0}\otimes L^n))
\to H^\bullet (X, \bR \sigma _*(\Omega ^i _{\cX}\otimes L^n))$$ is
an isomorphism.

\begin{lemma} Let $Y$ be a smooth variety, $B\subset Y$ - a smooth
divisor, and $M$ - the corresponding line bundle. Then for each
$i\geq 1$ we have the exact sequences
$$0\to \Omega ^i_Y\to M\otimes \Omega ^i_Y\to M\otimes \Omega
^i_Y\otimes \cO _B\to 0,$$
$$0\to \Omega _B^{i-1}\to M\otimes \Omega ^i_Y\otimes \cO _B\to
M\otimes \Omega ^i_B \to 0.$$ These sequences are functorial with
respect to the pair $(Y,B).$
\end{lemma}

\begin{proof} \cite{LNM1335},p.136.
\end{proof}

Let $D\subset X$ be a general divisor corresponding to $L^n$ for
$n\geq 1.$ Let
\begin{equation}\label{main-diag-divisor}\begin{array}{lcl}
\cZ & \stackrel{\pi}{\longrightarrow} & \cW \\
\downarrow \tilde{\sigma}& & \downarrow \sigma\\
\cZ _0 & \stackrel{\pi _0}{\longrightarrow} & D
\end{array}
\end{equation}
be the restriction of the diagram \ref{main-diag} to $D.$ Since $D$
is general this diagram has similar properties: $\cW$ is a smooth
projective poset scheme, $\pi _0$ is a hyperresolution, and for each
scheme $W_\alpha \in \cW$ the induced morphism $\pi :\pi
^{-1}(W_\alpha)\to W_{\alpha}$ is also a hyperresolution. Also the
adjunction morphism $\bbC _{D^{\an}}\to \bbR \sigma ^{\an}_*\bbC
_{\cW ^{\an}}$ is a quasi-isomorphism.

The exact sequences in the last lemma give rise to similar exact
sequences on poset schemes $\cX$ and $\cY _0$ respectively. Namely,
we have
\begin{equation}\label{d-1}\begin{array}{c}
0\to \Omega ^i_{\cY _0}\to L^n\otimes \Omega ^i_{\cY _0}\to
L^n\otimes
\Omega ^i_{\cY _0}\otimes \cO _{\cZ _0}\to 0,\\
0\to \Omega _{\cZ _0}^{i-1}\to L^n\otimes \Omega ^i_{\cY _0} \otimes
\cO _{\cZ _0}\to L^n\otimes \Omega ^i_{\cZ _0} \to 0, \end{array}
\end{equation} and
\begin{equation}\label{d-2}\begin{array}{c}
 0\to \Omega ^i_{\cX}\to L^n\otimes \Omega ^i_{\cX}\to L^n\otimes
\Omega ^i_{\cX}\otimes \cO _{\cW}\to 0,\\
0\to \Omega _{\cW}^{i-1}\to L^n\otimes \Omega ^i_{\cX} \otimes \cO
_{\cW}\to L^n\otimes \Omega ^i_{\cW} \to 0.\end{array}
\end{equation}

We now push forward these diagrams \ref{d-1} and \ref{d-2} by the
functors $\bR \pi _{0
*}$ and $\bR \sigma _*$ respectively. By functoriality we have a
morphism between the resulting exact triangles on $X.$ On the
hypercohomology this morphism induces an isomorphism in the term
$\Omega ^i_{\cY _0}$ by Lemma \ref{isom-hypercohom}. By induction it
also induces similar isomorphisms in the terms $\Omega _{\cZ
_0}^{i-1}$ and $L^n\otimes \Omega ^i_{\cZ _0}.$ Hence it induces an
isomorphism of hypercohomology in the term $L^n\otimes \Omega
^i_{\cY _0} \otimes \cO _{\cZ _0}$ and thus also in the term
$L^n\otimes \Omega ^i_{\cY _0}$ which proves the proposition and the
theorem.
\end{proof}
\end{proof}

\part{Appendix}

\section{Coherator and the functors $\bL f^*,\bR f_*$}

Probably this appendix contains nothing new but we decided to put
together some "well known" facts for convenience.

Let $X$ be a quasi-compact separated scheme. As usual $QcohX$
denotes the category of quasi-coherent sheaves on $X,$
$C(X)=C(QcohX)$ - the category of complexes over $QcohX,$
$D(X)=D(QcohX)$ - the derived category. We also consider the
category $\Mod _X$ of {\it all} $\cO _X$-modules, its category of
complexes $C(\Mod _X)$ and the corresponding derived category
$D(\Mod _X).$ Let $C_{\qc} (\Mod _X)\subset C(\Mod _X),$
$D_{\qc}(\Mod _X)\subset D(\Mod _X)$ be the full subcategories of
complexes with quasi-coherent cohomologies.

Both $Qcoh X$ and $\Mod _X$ are Grothendieck categories.

The obvious exact functor $\phi: Qcoh X\to \Mod _X$ preserves finite
limits and arbitrary colimits. It has a left-exact right adjoint
functor $Q_X=Q:\Mod _X\to QcohX$ - the {\it coherator}. The functor
$Q$ preserves arbitrary limits and injective objects. The induced
functor $Q:C(\Mod _X)\to C(X)$ preserves h-injectives. One defines
the right derived functor $\bR Q :D(\Mod _X)\to D(X)$ using the
h-injectives.

\begin{prop} \label{equivalences}
The functors $\phi,$ $\bR Q$ induce mutually inverse equivalences of
categories
$$\phi : D(X)\to D_{\qc}(\Mod _X),\quad \bR Q :D_{\qc}(\Mod _X)\to
D(X).$$
\end{prop}

\begin{proof} See for example [AlJeLi],Prop.1.3.
\end{proof}

\begin{lemma} \label{preserve-h-flats}
The functor $\phi :C(X)\to C(\Mod _X)$ preserves h-flats.
\end{lemma}

\begin{proof} Let $F\in C(X)$ be h-flat, $N\in C(\Mod _X)$ be
acyclic, $x\in X.$ We need to show that the complex of $\cO
_x$-modules $(F\otimes _{\cO _X}N)_x=F_x\otimes _{\cO _x}N_x$ is
acyclic. Let $i:Spec \cO _x\to X$ be the inclusion and
$\tilde{N}_x\in C(Qcoh (Spec \cO _x))$ be the sheafification of the
acyclic complex $N_x$ of $\cO _x$-modules. Then $i_*\tilde {N}_x$ is
an acyclic complex of quasi-coherent sheaves on $X.$ Hence the
complex $F\otimes _{\cO _X}i_*\tilde{N}_x$ is also acyclic. Thus
$F_x\otimes _{\cO _x}N_x=(F\otimes _{\cO _X}i_*\tilde {N}_x)_x$ is
also acyclic.
\end{proof}

Let $f:X\to Y$ be a quasi-compact separated morphism of
quasi-compact separated schemes. One defines the derived functors
$$\bL f^* :D(\Mod _Y)\to D(\Mod _X), \quad \bR f_*:D(\Mod _X)\to
D(\Mod _Y),$$ using h-flats and h-injectives in $C(\Mod _Y)$ and
$C(\Mod _X)$ respectively  [Sp].

We can also define the derived functor $\bL f^*:D(Y)\to D(X)$ using
the h-flats in $C(Y)$ (There are enough h-flats in $C(Y)$
[AlJeLi],Prop.1.1).

\begin{lemma} \label{commute-inverse-image}
There exists a natural isomorphism of functors
 $$\bL f^*\cdot \phi _Y = \phi _X \cdot \bL f^*:D(Y)\to D(\Mod _X).$$
\end{lemma}

\begin{proof}  Let $F\in D(Y)$ be h-flat. Then $\phi _X \cdot \bL
f^*(F)=\phi _X\cdot f^*(F).$ On the other hand $\phi _Y(F)$ is
h-flat by Lemma \ref{preserve-h-flats}. Hence $\bL f^*\cdot \phi _Y
(F)=f^* \cdot \phi _Y(F)=\phi _X\cdot f^*.$
\end{proof}

\begin{prop} a). The functors $(\bL f^*, \bR f_*)$ between $D(\Mod
_Y)$ and $D(\Mod _X)$ are adjoint.

b). These functors preserve the subcategories $D_{\qc}(\Mod _Y)$ and
$D_{\qc}(\Mod _X).$
\end{prop}

\begin{proof} a). It is [Sp],Prop.6.7. b).
For the functor $\bL f^* $ it follows from Proposition
\ref{equivalences} and Lemma \ref{commute-inverse-image}  and for
the functor $\bR f_*$ it is proved for example in [BoVdB],Thm.3.3.3
for the functor $\bR f_*.$
\end{proof}

The functors $f^*:QcohY\to QcohX,$ $f_*:QcohX\to QcohY$ are well
defined and clearly $f^*\cdot \phi _Y =\phi _X \cdot f^*.$ Hence
also $f_*\cdot Q_X=Q_Y\cdot f_*$ by adjunction. One defines the
derived functor
$$ \bR f_*:D(X)\to D(Y)$$
using  h-injectives in  $C(X).$

\begin{prop} \label{coherator-direct-image}
There exist a natural isomorphism of functor
$$\bR f_* \cdot \bR Q_X\simeq \bR Q_Y \cdot \bR f_*:D_{\qc}(\Mod
_X)\to D(Y).$$
\end{prop}

\begin{proof}  Let $I\in D_{\qc}(\Mod _X)$ be h-injective. Then $\bR
Q_X(I)=Q_X(I)$ is h-injective in $D(X).$ Hence $\bR f_*\cdot \bR
Q_X(I)=f\cdot Q_X(I).$ also $\bR f_*(I)=f_*(I).$ Since $f\cdot
Q_X(I)=Q_Y\cdot f(I)$ we get a morphism of functors
$$\theta :\bR f_* \cdot \bR Q_X \to \bR Q_Y \cdot \bR f_*.$$
We claim that $\theta $ is an isomorphism, i.e. $Q_Y \cdot
f_*(I)\simeq  \bR Q_Y \cdot f_*(I).$ We will use a lemma.

\begin{lemma} The functors $\bR f_*:D_{\qc}(\Mod _X)\to D_{\qc}(\Mod _Y),$
$\bR f_*:D(X)\to D(Y),$ and $\bR Q$ are way-out in both directions
(\cite{Ha}).
\end{lemma}

\begin{proof} Obviously all three functors are way-out left. The functor
$\bR f_*:D_{\qc}(\Mod _X)\to D_{\qc}(\Mod _Y)$ is way-out right by
[Li] (see also [BoVdB], Thm.3.3.3). For the functor $\bR Q$ see for
example the proof of Proposition 1.3 in [AlJeLi].

Let us prove that the functor $\bR f_*:D(X)\to D(Y)$ is way out
right. We may assume that $Y$ is affine and hence $f_*(-)=\Gamma
(X,-).$

Choose a finite affine open covering $\cU =\{ U_i\}_{i=1}^n$ of $X.$
For $F\in C(X)$ denote by
$$C_{\cU}(F):=0\to \oplus _{\vert I\vert =1}F_I\to
\oplus _{\vert I\vert =2}F_I \to ...$$ the corresponding (finite)
Cech resolution $F$ by alternating cochains. Here $I\subset
\{1,...,n\},$ $i:\cap _{i\in I}U_i\to X$ and $F_I=i_*i^*F\in C(X).$
The complex $F$ is quasi-isomorphic to $C_{\cU}(F).$ Notice that
each complex $F_I$ is acyclic for $\Gamma (X,-),$ i.e. $\bR \Gamma
(X,F_I)=\Gamma (X,F_I).$ This shows that if $F$ is in $D^{\leq
0}(X),$ then $\bR f_*F\in D^{\leq n-1}(Y).$
\end{proof}

Using the lemma it suffices to prove that $\theta (M)$ is an
isomorphism for a single quasi-coherent sheaf $M$ on $X$
(\cite{Ha},Ch.1,Prop.7.1,(iii)). In other words we may assume that
$I$ is an (bounded below) injective resolution in $\Mod _X$ of $\phi
(M)$ for $M\in Qcoh X.$ Then $Q_X(I)$ is an injective resolution of
$M$ in $QcohX.$ So $Q_Y\cdot f_*(I)=f_*\cdot Q_X(I)$ computes the
derived direct image of $M$ in the category of quasi-coherent
sheaves. On the other hand $f_*(I)$ computes the derived direct
image of $\phi (M).$ Since $f_*(I)\in D_{\qc}(\Mod _Y)$ it is
quasi-isomorphic to $\bR Q_Y\cdot f_*(I).$ So the needed assertion
becomes $\bR f_*(M)\simeq \bR f_*\cdot \phi (M).$ This is proved for
example in [ThTr],Appendix B,B.10.
\end{proof}

\begin{cor}\label{maps-to-h-inj} Let $I \in C(X)$ be h-injective and
$F\in C(Y)$ be h-flat. Then $$\Hom (F,f_*(I))=\Hom
_{D(X)}(F,f_*(I)).$$
\end{cor}

\begin{proof} An analogous statement for the category $D(\Mod _X)$
is proved in \cite{Sp}.

We may assume that $I=Q_X(J)$ for an h-injective $J\in D(\Mod _X).$
Then
$$\Hom (F,f_*\cdot Q_X(J))=\Hom (F,Q_Y\cdot f_*(J))=\Hom (\phi (F),
f_*(J)).$$ Since $\phi (F)$ is h-flat (Lemma \ref{preserve-h-flats})
by [Sp] we have
$$\Hom (\phi (F),f_*(J))=\Hom _{D(\Mod _Y)}(\phi (F),f_*(J)),$$
and by adjunction $\Hom _{D(\Mod _Y)}(\phi (F),f_*(J))= \Hom _{D(X)}
(F,\bR Q_Y\cdot f_*(J)).$ But in the proof of Proposition
\ref{coherator-direct-image} we established a quasi-isomorphism $\bR
Q_Y\cdot f_*(J)\simeq Q_Y\cdot f_*(J).$ This proves the lemma.
\end{proof}

\begin{cor} The functors $\bL f^*:D(Y)\to D(X)$ and $\bR f_*
:D(X)\to D(Y)$ are adjoint.
\end{cor}

\begin{proof} It follows immediately from Corollary
\ref{maps-to-h-inj}.
\end{proof}

\end{document}